\documentclass[11pt, leqno, letterpaper]{amsart}
 
\usepackage{graphicx}    % needed for including graphics e.g. EPS, PS

\usepackage{amsmath,amsthm,amsfonts,amssymb}

\newtheorem{theorem}{Theorem}

\newtheorem{definition}[theorem]{Definition}

\newtheorem{lemma}[theorem]{Lemma}

\newtheorem{proposition}[theorem]{Proposition}

\newtheorem{remark}[theorem]{Remark}

\begin{document}         
 % Start your text

\title[Growth Rate of a Linear Stochastic Recursion]{On the growth rate of a linear 
stochastic recursion with Markovian dependence}

\author{Dan Pirjol}
\email{
dpirjol@gmail.com}

\author{Lingjiong Zhu}
\email{
zhul@umn.edu}

\date{10 November 2014. \textit{Revised:} 11 May 2015}

  %stochastic processses, special processes, phase transitions, large deviations 

\begin{abstract}
We consider the linear stochastic recursion  $x_{i+1} = a_{i}x_{i}+b_{i}$ where the 
multipliers $a_i$ are random and have Markovian dependence given by the exponential
of a standard Brownian motion and $b_{i}$ are i.i.d. positive random noise independent
of $a_{i}$. Using large deviations theory we study the growth rates (Lyapunov
exponents) of the positive integer moments 
$\lambda_q = \lim_{n\to \infty} \frac{1}{n} \log
\mathbb{E}[(x_n)^q]$ with $q\in \mathbb{Z}_+$.
We show that the Lyapunov exponents $\lambda_q$ exist, under appropriate
scaling of the model parameters,
and have non-analytic behavior manifested as a phase 
transition. We study the properties of the phase transition and the critical 
exponents using both analytic and numerical methods. 
\end{abstract}

\maketitle

\section{Introduction}
\label{intro}

Random recursions are widely used to model processes in economics, biology,
computer science and physics, 
see e.g. \cite{Kadanoff,Kesten,Lewontin,Mitzenmacher,Sornette}. 
The best known is the Kesten process which is
defined by the linear stochastic recursion \cite{Kesten}
\begin{eqnarray}\label{Kesten}
x_{i+1} = a_i x_i + b_i
\end{eqnarray}
where $a_i, b_i$ are random i.i.d. real variables with $a_i>0$. 
The distributional properties of this process and the conditions for 
the existence of stationary and ergodic distributions are well-known \cite{Kesten,Vervaat,Goldie}. 
The process (\ref{Kesten}) has a unique stationary solution, provided
that $\mathbb{E}[\log a_i] < 0$ and $\mathbb{E}[\log (b_i)_+] < \infty$.
Under these conditions the stationary distribution has power-like 
tails (more precisely of regular
variation) $\mathbb{P}(x_n > X) \sim X^{-\mu}$ for $X\to \infty$, where the 
exponent $\mu$ is determined by the equation $\mathbb{E}[a_{i}^{\mu}]=1$. 
See \cite{MST} for a recent overview of the properties of this process. 

The generalization of the process (\ref{Kesten}) to the case of random
coefficients $a_i,b_i$ with Markovian dependence has been also considered
in the mathematical literature \cite{Roitershtein,Saporta}.

In this paper we consider a linear stochastic recursion of the form (\ref{Kesten})
with stochastic multipliers $a_i$ following a Markov process
given by the exponential of a standard Brownian motion. 
We consider the stochastic recursion defined by
\begin{eqnarray}\label{originalEqn}
x_{i+1} = a_i x_i+b_{i}\,,\quad
a_i = 1 + \rho e^{\sigma W_i - \frac12\sigma^2 t_i}
\end{eqnarray}
with initial condition $x_0>0$ and $b_{i}$ are i.i.d. positive random variables 
independent of $(a_{i})_{i=0}^{\infty}$.
$W_i$ is a standard Brownian motion sampled on the uniformly spaced
times $t_i$ with time step $t_{i+1} - t_i = \tau$ and $t_{0}=0$, i.e. $W_{i}=W(t_{i})$, 
where $W(t)$ is a standard Brownian motion starting at $0$ at time $0$. 
The parameters $\sigma,\rho \in \mathbb{R}_+$ are real positive numbers.
Although superficially similar to the linear recursion (\ref{Kesten}), the 
process (\ref{originalEqn}) has a different qualitative from the Kesten process 
as it is non-stationary. 

For the moment, we will assume the homogeneous case $b_{i}\equiv 0$, and we 
consider the random multiplicative process
\begin{eqnarray}\label{RMPdef}
x_{i+1} = a_i x_i\,,\quad
a_i = 1 + \rho e^{\sigma W_i - \frac12\sigma^2 t_i}
\end{eqnarray}
The process \eqref{originalEqn} under the full stated assumptions will be 
considered in Section \ref{LinearSection}.
The process (\ref{RMPdef}) is a stochastic growth process with correlated 
multipliers $a_i > 1$. It can be regarded as a discrete time version of the 
continuous-time process $dX_t = 
\rho e^{\sigma W_t - \frac12 \sigma^2 t} X_t dt$ by writing it as
\begin{eqnarray}
x_{i+1} - x_i = \rho e^{\sigma W_i - \frac12 \sigma^2 t_i} x_i 
\end{eqnarray}
The process for $X_t$ is solved
in terms of the time integral of the geometric Brownian motion
\begin{eqnarray}
X_t = X_0 \exp \Big( \rho \int_0^t ds e^{\sigma W_s - \frac12 \sigma^2 s}\Big)\,.
\end{eqnarray}
The time integral of the geometric Brownian motion appears in many problems
of statistics, probability and mathematical finance, and its distributional properties
have been widely studied \cite{Dufresne1,Dufresne2,Yor}. 
We will show that the discrete time version (\ref{RMPdef}) of this process 
has new and interesting properties, which are qualitatively different from 
those of the continuous time case. 

The process (\ref{RMPdef}) appears in several mathematical finance problems.
For example, it describes the bank account under discrete time compounding with
interest rates following a geometric Brownian motion. This corresponds to the
money market account in the Black-Derman-Toy model of stochastic interest rates
\cite{BDT}. This quantity plays a central role in the simulation of a short rate
model, as it gives the numeraire in the risk-neutral measure. 
A good understanding of its distributional properties is crucial for the 
numerical application of the model.

The process (\ref{RMPdef}) can be also related to the statistical mechanics of 
disordered systems, see \cite{BG} for an overview. The random variable $x_n$ 
is related to the grand partition function of a one-dimensional
lattice gas of non-interacting particles placed in a random external 
field $\phi_i$ given by a Brownian motion. This is given by
\begin{eqnarray}
\mathcal{Z}[\phi] = \prod_{k=1}^n (1 + e^{-\beta \phi_i + \beta \mu})\,,
\end{eqnarray}
with $\beta =1/T$ the inverse temperature and $\mu$ the chemical potential.
The partition function $\mathcal{Z}[\phi]$ has the same general form as 
$x_n$ given by the process (\ref{RMPdef}). Such systems have been 
studied in \cite{BK,CM,CMY} on the case of a continuous one-dimensional gas
in the quenched random field approximation.

We would like to study the large $n$ asymptotics of the growth rate of 
the positive integer moments $\mathbb{E}[(x_n)^q]$ with $q\in \mathbb{Z}_+$. 
More precisely, we will consider the $n\to \infty$ limit at fixed
\begin{eqnarray}\label{betadef}
\beta = \frac12 \sigma^2 t_n n = \frac12\sigma^2 \tau n^2\,.
\end{eqnarray}
Using large deviations theory we will show that the following limit exists 
and defines a Lyapunov exponent for the growth rate of the expectation 
\begin{eqnarray}\label{Lyapunov_def}
\lambda(\rho,\beta;q) = \lim_{n\to \infty} \frac{1}{n} \log \mathbb{E}[(x_n)^q]\,,\quad
q\in \mathbb{Z}_+\,.
\end{eqnarray}

The process (\ref{RMPdef}) was considered in \cite{DPI}, where the
positive integer moments $\mathbb{E}[(x_n)^q]$ have been computed exactly (but numerically)
for finite $n$. This numerical study showed that
these moments have a sharp explosion for sufficiently large $n$ or $\sigma$, which 
signals the appearance of heavy tailed distributions for the random variable $x_n$. 
This phenomenon can be studied by mapping the problem to a one-dimensional lattice gas,
and the explosion is related to a gas-liquid phase transition in this system \cite{DPII}.

The analyticity of Lyapunov exponents has been extensively studied for products
of random matrices with i.i.d. distributed matrix elements, see \cite{CN,Ruelle}.
The problem considered here is somewhat different, as the Lyapunov exponent 
refers to a deterministic quantity, the expectation of the random variable $x_n$. 

The paper is organized as follows. In Section \ref{LDPSection}, we use large 
deviations theory to prove the existence of the Lyapunov exponent and express 
it as a variational problem. The convergence of the random variable 
$\frac{1}{n}\log x_n$ as $n\to \infty$ and its fluctuations are studied in 
Section \ref{ASSection}.
We analyze the variational problem for the Lyapunov exponent in Section 
\ref{VariationalSection}, and 
present a numerical study of its solution in Section \ref{NumericalSection}.
Section \ref{MeanFieldSection} summarizes an approximation to the solution 
of the variational problem which is equivalent to the mean field approximation. 
A rigorous study of the phase transitions is presented in 
Section \ref{PhaseSection}. The slope of the phase transition curve  
and the critical exponents are studied in Sections \ref{Sec:slope} and
\ref{CriticalSection}, respectively. 
Generalizations to the Lyapunov exponents of the positive integer moments
$q\geq 2$ and the linear stochastic recursion will be studied in Section 
\ref{Sec:10} respectively. 

\section{Lyapunov Exponent and Large Deviations}
\label{LDPSection}

Let us consider the discrete time random multiplicative process defined by
\begin{equation}\label{RMPdef0}
x_{i+1}=a_{i}x_{i},
\qquad
a_{i}=1+\rho e^{\sigma W_{i}-\frac{1}{2}\sigma^{2}t_{i}},
\end{equation}
where $W_{i}=W(t_{i})$ is the value of standard Brownian motion at time $t_{i}$
and $t_{i}=i\tau$. The parameters $\rho$ and $\sigma$ are real positive constants.

Define the Lyapunov exponent $\lambda(\rho,\beta)$ as
\begin{equation}
\lambda(\rho,\beta) := \lim_{n\rightarrow\infty}\frac{1}{n}\log\mathbb{E}[x_{n}]\,.
\end{equation}
We would like to show the existence of the Lyapunov exponent and express it in 
terms of a variational formula. Our proof uses the large deviations theory from 
probability.
Before we proceed, 
recall that a sequence $(P_{n})_{n\in\mathbb{N}}$ of probability measures on a topological space $\mathbb{X}$ 
satisfies the large deviations principle with rate function $\mathcal{I}:\mathbb{X}\rightarrow\mathbb{R}$ if $\mathcal{I}$ is non-negative, 
lower semicontinuous and for any measurable set $A$, we have
\begin{equation}
-\inf_{x\in A^{o}}\mathcal{I}(x)\leq\liminf_{n\rightarrow\infty}\frac{1}{n}\log P_{n}(A)
\leq\limsup_{n\rightarrow\infty}\frac{1}{n}\log P_{n}(A)\leq-\inf_{x\in\overline{A}}\mathcal{I}(x).
\end{equation}
Here, $A^{o}$ is the interior of $A$ and $\overline{A}$ is its closure. 
We refer to \cite{Dembo} and \cite{VaradhanII} 
for general background of the theory and the applications
of large deviations.

\begin{theorem}\label{LyapunovThm}
The limit $\lambda(\rho,\beta):=\lim_{n\rightarrow\infty}\frac{1}{n}\log\mathbb{E}[x_{n}]$ exists
and it can be expressed as a variational formula
\begin{equation}\label{theorem1}
\lambda(\rho,\beta)
=\sup_{g\in\mathcal{G}}\left\{\log\rho g(1)+\beta\int_{0}^{1}(g(1)-g(x))^{2}dx-\int_{0}^{1}I(g'(x))dx\right\},
\end{equation}
where $I(x):=x\log x+(1-x)\log(1-x)$ and
\begin{equation}
\mathcal{G}:=\left\{g:[0,1]\rightarrow[0,1]: \text{$g(0)=0$, $g$ is absolutely continuous and $0\leq g'\leq 1$}\right\}.
\end{equation}
\end{theorem}

\begin{proof}
By iterating the expression $x_{i+1}=a_{i}x_{i}$, we get
\begin{align}
\mathbb{E}[x_{n}]&=\mathbb{E}[a_{n-1}a_{n-2}\cdots a_{1}a_{0}x_{0}]
\\
&=x_{0}\mathbb{E}\left[\prod_{i=0}^{n-1}\left(1+\rho e^{\sigma W_{i}-\frac{1}{2}\sigma^{2}t_{i}}\right)\right]
\nonumber
\\
&=x_{0}\mathbb{E}\left[\prod_{i=0}^{n-1}\left(1+e^{\log\rho+\sigma W_{i}-\frac{1}{2}\sigma^{2}t_{i}}\right)\right]
\nonumber
\\
&=x_{0}2^{n}\mathbb{E}\left[\prod_{i=0}^{n-1}\left(\frac{1}{2}+\frac{1}{2}e^{\log\rho+\sigma W_{i}-\frac{1}{2}\sigma^{2}t_{i}}\right)\right]
\nonumber
\\
&=x_{0}2^{n}\mathbb{E}\left[\prod_{i=0}^{n-1}e^{(\log\rho+\sigma W_{i}-\frac{1}{2}\sigma^{2}t_{i})Y_{i}}\right],
\nonumber
\end{align}
where $Y_{i}$ are i.i.d. random variables taking values $1$ and $0$ with probability $1/2$ and independent of the Brownian motion $W(t)$.
Notice that $W_{i}$ are dependent random variables. However, their increments are independent. 
Therefore, it is natural to define their increments as 
$V_{i}=W_{i+1}-W_{i},\quad i=0,1,2,3,\ldots$
and therefore, we get
\begin{align}
\mathbb{E}[x_{n}]&=x_{0}2^{n}\mathbb{E}\left[\prod_{i=0}^{n-1}e^{(\log\rho+\sigma W_{i}-\frac{1}{2}\sigma^{2}t_{i})Y_{i}}\right]
\\
&=x_{0}2^{n}\mathbb{E}\left[e^{\log\rho\sum_{i=0}^{n-1}Y_{i}+\sum_{i=0}^{n-1}\sum_{j=0}^{i-1}(\sigma V_{j}-\frac{1}{2}\sigma^{2}\tau)Y_{i}}\right]
\nonumber
\\
&=x_{0}2^{n}\mathbb{E}\left[e^{\log\rho\sum_{i=0}^{n-1}Y_{i}+\sum_{j=0}^{n-2}(\sum_{i=j+1}^{n-1}Y_{i})(\sigma V_{j}-\frac{1}{2}\sigma^{2}\tau)}\right]
\nonumber
\end{align}
It is clear that $\sigma V_{j}-\frac{1}{2}\sigma^{2}\tau$ are i.i.d. Gaussian random variables with mean
$-\frac{1}{2}\sigma^{2}\tau$ and variance $\sigma^{2}\tau$. Therefore, 
\begin{align}\label{3Terms}
\mathbb{E}[x_{n}]
&= x_{0}2^{n}\mathbb{E}\left[e^{\log\rho\sum_{i=0}^{n-1}Y_{i}-
\frac12 \sigma^{2}\tau\sum_{j=0}^{n-2}(\sum_{i=j+1}^{n-1}Y_{i})
+\frac{1}{2}\sigma^{2}\tau\sum_{j=0}^{n-2}(\sum_{i=j+1}^{n-1}Y_{i})^{2}}\right]  \\
&= x_0 2^n \mathbb{E}\left[ e^{A_n} \right] \nonumber
\end{align}
where the exponent $A_n$ is given by
\begin{align}
A_n &=
\log\rho\sum_{i=1}^{n}Y_{i}+\frac{\beta}{n^{2}}
\sum_{j=0}^{n-1}\left(\sum_{i=j+1}^{n}Y_{i}\right)^{2} + r_n 
 \nonumber\\
&= 
\log\rho\sum_{i=1}^{n}Y_{i}+\frac{\beta}{n}\int_{0}^{1}
\left(\sum_{i=\lfloor nx\rfloor+1}^{n}Y_{i}\right)^{2}dx + r_n \nonumber \\
&=
n\left[\log\rho\left(\frac{1}{n}\sum_{i=1}^{n}Y_{i}\right)
+\beta\int_{0}^{1}\left(\frac{1}{n}\sum_{i=1}^{n}Y_{i}-\frac{1}{n}
\sum_{i=1}^{\lfloor nx\rfloor}Y_{i}\right)^{2}dx\right] + o(n)\,. \nonumber
\end{align}
The correction $r_n$ is bounded by a deterministic constant and is 
negligible in the large $n$ limit.
Recalling that we are interested in the $n\to \infty$ limit at fixed $\beta = 
\frac{1}{2}\sigma^{2}\tau n^{2}$, it is easy to see that the middle term in 
\eqref{3Terms} satisfies
\begin{equation}
0\leq \frac12
\sigma^{2}\tau\sum_{j=0}^{n-2}\left(\sum_{i=j+1}^{n-1}Y_{i}\right)\leq
\frac{\beta}{n^{2}}n^{2}=\beta,
\end{equation}
and it is negligible.

By large deviations theory in probability, the Mogulskii theorem (see Theorem 5.1.2 in
\cite{Dembo}) says that
$\mathbb{P}(\frac{1}{n}\sum_{i=1}^{\lfloor n\cdot\rfloor}Y_{i}\in\cdot)$ satisfies a sample path large deviations principle
on the space $L_{\infty}[0,1]$ (i.e. the space of functions on $[0,1]$ equipped with supremum norm)
with the rate function
\begin{equation}
\int_{0}^{1}I(g'(x))dx,
\end{equation}
where $g(0)=0$, $g$ is absolutely continuous, $0\leq g'\leq 1$, and the rate function is $+\infty$ 
otherwise and $I(x)=x\log x+(1-x)\log(1-x)+\log 2$ is a relative entropy function.
Informally speaking, it says that
\begin{equation}
\mathbb{P}\left(\frac{1}{n}\sum_{i=1}^{\lfloor nx\rfloor}Y_{i}\simeq g(x), 0\leq x\leq 1\right)
\simeq e^{-n\int_{0}^{1}I(g'(x))dx+o(n)},
\end{equation}
as $n\rightarrow\infty$.

In large deviations theory, the celebrated Varadhan's lemma says that if $P_{n}$ satisfies
a large deviations principle with rate function $\mathcal{I}(x)$ on $\mathbb{X}$ and $F:\mathbb{X}\rightarrow\mathbb{R}$
is a bounded and continuous function, then
\begin{equation}
\lim_{n\rightarrow\infty}\frac{1}{n}\log\int_{\mathbb{X}}e^{nF(x)}dP_{n}(x)=\sup_{x\in\mathbb{X}}\{F(x)-\mathcal{I}(x)\}.
\end{equation}
It is easy to check that for any $g\in L_{\infty}[0,1]\cap\mathcal{G}$,
\begin{equation}
g\mapsto\log\rho\cdot g(1)+\beta\int_{0}^{1}(g(1)-g(x))^{2}dx
\end{equation}
is a bounded and continuous map.

Hence, by Varadhan's lemma, we conclude that the limit $\lim_{n\to \infty}
\frac{1}{n} \log \mathbb{E}[x_n]$ exists and is given by the variational
problem (\ref{theorem1}).
\end{proof}

%We will discuss in the next section the solution of the variational problem
%for $\lambda(\rho,\beta)$ following from this theorem.
We present next a few immediate implications of this result for the asymptotics
of the Lyapunov exponent in the limiting cases $\beta\to 0, \infty$ and
$\rho \to \infty$.

\begin{proposition}\label{prop:bounds}
(i) The $\beta \to 0$ limit for the Lyapunov exponent is
\begin{eqnarray}
\lambda(\rho,0) = \log(\rho+1) 
\end{eqnarray}

(ii) The Lyapunov exponent is bounded from above and below as 
\begin{eqnarray}\label{bound1}
\frac13\beta + \log(1 + \rho) \geq \lambda(\rho,\beta) \geq 
\frac13 \beta + \log\rho\,.
\end{eqnarray}

(iii) These bounds give the asymptotic behavior in the 
large $\rho$ and large $\beta$ limits
\begin{equation}\label{LargeBeta}
\lim_{\beta\rightarrow\infty}\frac{\lambda(\rho,\beta)}{\beta}=\frac{1}{3},
\end{equation}
and
\begin{equation}\label{LargeRho}
\lim_{\rho\rightarrow\infty}\left|\lambda(\rho,\beta)-\frac{\beta}{3}-\log\rho\right|=0.
\end{equation}
\end{proposition}

\begin{proof}
i) Taking into account that $x\to I(x)$ is a convex function, 
Jensen's inequality implies the upper bound
\begin{align}
\lambda(\rho,0)&\leq\sup_{g}\left\{\log\rho g(1)-I\left(\int_{0}^{1}g'(x)dx\right)\right\}
\\
&=\sup_{g}\left\{\log\rho g(1)-I(g(1))\right\}
\nonumber
\\
&=\sup_{0\leq x\leq 1}\left\{\log\rho x-I(x)\right\}. \nonumber
\end{align}
On the other hand, choosing $g(x)=g(1)x$, it is clear that we have the lower bound
\begin{equation}
\lambda(\rho,0)\geq\sup_{0\leq x\leq 1}\left\{x\log\rho-I(x)\right\}.
\end{equation}
Hence, when $\beta=0$, we must have
\begin{equation}
\lambda(\rho,0)=\sup_{0\leq x\leq 1}\left\{x\log\rho-I(x)\right\}.
\end{equation}
At optimality, $\rho=\frac{x}{1-x}$ and therefore
\begin{equation}
\lambda(\rho,0)=\log(\rho+1).
\end{equation}

ii) A lower bound on the Lyapunov exponent $\lambda(\rho,\beta)$ can be obtained 
by taking $g(x)=g(1)x$ in (\ref{theorem1}), which gives
\begin{equation}
\lambda(\rho,\beta)\geq
\sup_{0\leq x\leq 1}\left\{\log\rho x+\frac{\beta}{3}x^{2}-I(x)\right\}\,.
\end{equation}
An explicit result for this lower bound will be given below in 
Section~\ref{MeanFieldSection}, see Proposition~\ref{prop7}. 
For now we derive a simpler but weaker lower bound by taking $g(1)=1$,
\begin{equation}
\lambda(\rho,\beta)\geq
\log\rho+\beta\int_{0}^{1}(1-x)^{2}dx-I(1)=\frac{\beta}{3}+\log\rho\,,
\end{equation}
which gives the lower bound in (\ref{bound1}).

On the other hand, by the Mean Value Theorem, 
$|g(1)-g(x)|\leq|1-x|$ for any $0\leq g'\leq 1$. Also, using the Jensen
inequality for the last term as in (i), we get the upper bound
\begin{align}
\lambda(\rho,\beta)&=
\sup_{g(0)=0,0\leq g'\leq 1}
\left\{\log\rho g(1)+\beta\int_{0}^{1}(g(1)-g(x))^{2}dx
-\int_{0}^{1}I(g'(x))dx\right\}
\\
&\leq\sup_{g(0)=0,0\leq g'\leq 1}
\left\{\log\rho g(1)+\beta\int_{0}^{1}(1-x)^{2}dx
-\int_{0}^{1}I(g'(x))dx\right\}
\nonumber
\\
&=\frac{\beta}{3}+\log(\rho+1).\nonumber
\end{align}
This proves the upper bound in (\ref{bound1}).

Dividing by $\beta$ in (\ref{bound1})
and taking the $\beta\to \infty$ limit we conclude that 
\begin{equation}
\lim_{\beta\rightarrow\infty}\frac{\lambda(\rho,\beta)}{\beta}=\frac{1}{3}\,.
\end{equation}
This proves the relation (\ref{LargeBeta}).

Subtracting $\frac13\beta + \log\rho$ in the inequalities (\ref{bound1}) gives
the asymptotics (\ref{LargeRho})
\begin{equation}
\lim_{\rho\rightarrow\infty}\left|\lambda(\rho,\beta)-
\frac{\beta}{3}-\log\rho\right|=0 \,.
\end{equation}
\end{proof}

\begin{remark}
The result (i) agrees with the intuitive expectation: the limit $\beta\to 0$
corresponds to taking $\sigma \to 0$ in the random multiplicative process 
(\ref{RMPdef0}), which becomes a deterministic recursion in this limit.
This is solved as $x_n = (1+\rho)^{n}$, and the Lyapunov exponent
is given immediately by $\lambda(\rho,0)  = \log(1+\rho)$.
\end{remark}

\section{Almost Sure Limit and Fluctuations}
\label{ASSection}

We prove in this section a strong Law of Large Numbers and a fluctuations
result for $\frac{1}{n} \log x_n$. We start by considering first the
simpler case of the random multiplicative model $x_{i+1}=a_i x_i$ and
then we will derive the LLN under the more general assumption of the
presence of i.i.d. additive noise $b_i$. 

%%%%%%%%%%%%%%%%%%%%%%%%%  Lemma  %%%%%%%%%%%%%%%%%%%%%%%%%
\begin{proposition}\label{AlmostSurelyThm}
Almost surely,
\begin{equation}
\lim_{n\rightarrow\infty}\frac{1}{n}\log x_{n}
=\lim_{n\rightarrow\infty}\frac{1}{n}[\log a_{n-1}+\log a_{n-2}+\cdots+\log a_{0}+\log x_{0}]
=\log(1+\rho).
\end{equation}
\end{proposition}

\begin{proof}
Let us recall that
\begin{equation}
x_{n}=a_{n-1}a_{n-2}\cdots a_{1}a_{0}x_{0},
\end{equation}
where
\begin{equation}
a_{i}=1+\rho e^{\sigma W_{i}-\frac{1}{2}\sigma^{2}t_{i}},
\end{equation}
where $t_{i}=i\tau$ and $\beta=\frac{1}{2}\sigma^{2}\tau n^{2}$ is a universal constant.
Therefore, it is easy to check that $\frac{1}{2}\sigma^{2}t_{n}=\frac{\beta}{n}\rightarrow 0$
as $n\rightarrow\infty$.

Moreover, for any $\epsilon>0$, by Chebyshev's inequality, for any $\theta>0$,
\begin{equation}
\mathbb{P}(|\sigma W_{n}|\geq\epsilon)
=2\mathbb{P}(\sigma W_{n}\geq\epsilon)
\leq 2\mathbb{E}[e^{\theta\sigma W_{n}}]e^{-\theta\epsilon}
=2e^{\frac{1}{2}\theta^{2}\sigma^{2}\tau n}e^{-\theta\epsilon}
=2e^{\theta^{2}\frac{\beta}{n}}e^{-\theta\epsilon}.
\end{equation}
By choosing $\theta=\sqrt{n}$, we get
\begin{equation}
\mathbb{P}(|\sigma W_{n}|\geq\epsilon)
\leq 2e^{\beta}e^{-\sqrt{n}\epsilon},
\end{equation}
and hence $\sum_{n=1}^{\infty}\mathbb{P}(|\sigma W_{n}|\geq\epsilon)<\infty$
for any $\epsilon>0$. By Borel-Cantelli lemma, $\sigma W_{n}\rightarrow 0$ almost surely.
Therefore, $a_{n}\rightarrow 1+\rho$ as $n\rightarrow\infty$ almost surely.
\end{proof}

%%%%%%%%%%%%%%%%  LLN  %%%%%%%%%%%%%%%%%%%%%%%%%%%%%%%%%%%%%%%%%%%

%We can study the almost sure limit 
%$\lim_{n\rightarrow\infty}\frac{1}{n}\log[(x_{n})^{q}]$.
%Observe that $\log[(x_{n})^{q}]=q\log x_{n}$, thus it suffices to consider 
%the case $q=1$.
The main LLN result follows.

\begin{theorem}
Assume that $\mathbb{E}[b_{0}]<\infty$. Then,
\begin{equation}
\lim_{n\rightarrow\infty}\frac{1}{n}\log x_{n}=\log(1+\rho),
\end{equation}
almost surely as $n\rightarrow\infty$.
\end{theorem}

\begin{proof}
Since $x_{n}\geq x_{0}\prod_{i=0}^{n-1}a_{i}$, by Proposition \ref{AlmostSurelyThm}, 
\begin{equation}
\liminf_{n\rightarrow\infty}\frac{1}{n}\log x_{n}\geq\log(1+\rho),
\end{equation}
almost surely as $n\rightarrow\infty$. On the other hand,
\begin{equation}
x_{n}\leq(x_{0}+b_{0}+b_{1}+\cdots+b_{n-1})\prod_{i=0}^{n-1}a_{i}.
\end{equation}
Therefore,
\begin{equation}
\frac{1}{n}\log x_{n}\leq\frac{1}{n}\log n
+\frac{1}{n}\log\left(\frac{x_{0}}{n}+\frac{b_{0}+b_{1}+\cdots+b_{n-1}}{n}\right)
+\frac{1}{n}\log\sum_{i=0}^{n-1}a_{i}.
\end{equation}
Since $\mathbb{E}[b_{0}]<\infty$, by strong law of large numbers, $\frac{b_{0}+b_{1}+\cdots+b_{n-1}}{n}\rightarrow\mathbb{E}[b_{0}]$
almost surely as $n\rightarrow\infty$. By Proposition \ref{AlmostSurelyThm}, we conclude that
\begin{equation}
\limsup_{n\rightarrow\infty}\frac{1}{n}\log x_{n}\leq\log(1+\rho),
\end{equation}
almost surely as $n\rightarrow\infty$.
\end{proof}

\begin{remark}
First, we notice that the almost sure limit $\lim_{n\rightarrow\infty}\frac{1}{n}\log x_{n}=\log(1+\rho)$
is analytic everywhere in the phase plane $(\rho,\beta)$ while the Lyapunov 
exponent exhibits phase transitions. 
The almost sure limit has very different behavior
than the Lyapunov exponent $\lim_{n\rightarrow\infty}\frac{1}{n}\log\mathbb{E}[x_{n}]$. The difference
arises from the fact that the almost sure limit is determined by the 
\emph{typical events}, i.e. the law of large numbers,
while the Lyapunov exponent is determined by the \emph{rare events}, i.e. the 
large deviations.
\end{remark}

Next we present a fluctuation result. 

\begin{proposition}\label{CLTProp}
When $b_{i}\equiv 0$,
\begin{equation}
\frac{\log x_{n}-n\log(1+\rho)}{\sqrt{n}}
\rightarrow N\left(0,\frac{2\beta}{3}\frac{\rho^{2}}{(1+\rho)^{2}}\right),
\end{equation}
in distribution as $n\rightarrow\infty$.
\end{proposition}

\begin{proof}
When $b_{i}\equiv 0$, $x_{n}=x_{0}\prod_{i=0}^{n-1}a_{i}$ and $\frac{1}{n}\log x_{n}\rightarrow\log(1+\rho)$ a.s.
as $n\rightarrow\infty$. It is straightforward to compute that
\begin{align}
\frac{\log x_{n}-n\log(1+\rho)}{\sqrt{n}}
&=\frac{\log x_{0}}{\sqrt{n}}
+\frac{1}{\sqrt{n}}\sum_{i=0}^{n-1}\log\left(1+\frac{\rho(e^{\sigma W_{i}-\frac{1}{2}\sigma^{2}t_{i}}-1)}{1+\rho}\right)
\\
&=\frac{\log x_{0}}{\sqrt{n}}
+\frac{1}{\sqrt{n}}\sum_{i=0}^{n-1}\frac{\rho}{1+\rho}\sigma W_{i}
+\epsilon_{n},
\nonumber
\end{align}
where
\begin{equation}
\epsilon_{n}:=\frac{1}{\sqrt{n}}\sum_{i=0}^{n-1}
\left[\log\left(1+\frac{\rho(e^{\sigma W_{i}-\frac{1}{2}\sigma^{2}t_{i}}-1)}{1+\rho}\right)
-\frac{\rho}{1+\rho}\sigma W_{i}\right].
\end{equation}
For any $x>-1$, $\log(1+x)\leq x$. Thus,
\begin{align}
\epsilon_{n}
&\leq\frac{1}{\sqrt{n}}\sum_{i=0}^{n-1}
\left[\frac{\rho(e^{\sigma W_{i}-\frac{1}{2}\sigma^{2}t_{i}}-1)}{1+\rho}
-\frac{\rho}{1+\rho}\sigma W_{i}\right]
\\
&\leq\frac{1}{\sqrt{n}}\frac{\rho}{1+\rho}
\sum_{i=0}^{n-1}
\left[e^{\sigma W_{i}}-1-\sigma W_{i}\right]=:\overline{\epsilon}_{n}.
\nonumber
\end{align}
Note that $\overline{\epsilon}_{n}\geq 0$ and
\begin{align}
\mathbb{E}[\overline{\epsilon}_{n}]
&=\frac{1}{\sqrt{n}}\frac{\rho}{1+\rho}\sum_{i=0}^{n-1}\left[e^{\frac{1}{2}\sigma^{2}i\tau}-1\right]
\\
&=\frac{1}{\sqrt{n}}\frac{\rho}{1+\rho}\sum_{i=0}^{n-1}\left[e^{\frac{\beta i}{n^{2}}}-1\right]
\nonumber
\\
&=\frac{1}{\sqrt{n}}\frac{\rho}{1+\rho}\left[\frac{e^{\frac{\beta}{n}}-1}{e^{\frac{\beta}{n^{2}}}-1}-n\right]
\rightarrow 0,
\nonumber
\end{align}
as $n\rightarrow\infty$. Thus $\overline{\epsilon}_{n}\rightarrow 0$ in probability.

For any $x\in\mathbb{R}$, let $F(x):=\log\left(1+\frac{\rho(e^{x}-1)}{1+\rho}\right)-\frac{\rho}{1+\rho}x$.
Then, $F(0)=0$, 
\begin{align}
&F'(x)=\frac{\frac{\rho}{1+\rho}e^{x}}{1+\frac{\rho}{1+\rho}(e^{x}-1)}-\frac{\rho}{1+\rho},
\\
&F''(x)=\frac{\frac{\rho}{(1+\rho)^{2}}e^{x}}{(1+\frac{\rho}{1+\rho}(e^{x}-1))^{2}}.
\end{align}
Thus $F''(x)>0$ for any $x\in\mathbb{R}$ and hence $F'(x)$ is increasing. 
Note that $F'(0)=0$, thus $F'(x)<0$ for any $x<0$ and $F'(x)>0$ for any $x>0$. 
Since $F(0)=0$, we conclude that $F(x)\geq 0$ for any $x\in\mathbb{R}$.
Therefore,
\begin{equation}
\epsilon_{n}
\geq\frac{1}{\sqrt{n}}\sum_{i=0}^{n-1}\left[\sigma W_{i}-\frac{1}{2}\sigma^{2}t_{i}-\sigma W_{i}\right]
=\frac{-1}{\sqrt{n}}\sum_{i=0}^{n-1}\frac{\beta i}{n^{2}}\rightarrow 0,
\end{equation}
as $n\rightarrow\infty$. Hence, we conclude that $\epsilon_{n}\rightarrow 0$ in probability
as $n\rightarrow\infty$. Also, we have
$\frac{\log x_{0}}{\sqrt{n}}\rightarrow 0$ as $n\rightarrow\infty$.
Finally, notice that $\frac{1}{\sqrt{n}}\sum_{i=0}^{n-1}\frac{\rho}{1+\rho}\sigma W_{i}$ 
is a normal random variable with mean zero and variance
\begin{align}
\mbox{Var}\left[\frac{1}{\sqrt{n}}\sum_{i=0}^{n-1}\frac{\rho}{1+\rho}\sigma W_{i}\right]
&=\mbox{Var}\left[\frac{1}{\sqrt{n}}\sum_{i=1}^{n-1}\frac{\rho}{1+\rho}\sigma(n-i)(W_{i}-W_{i-1})\right]
\nonumber
\\
&=\frac{1}{n}\sum_{i=1}^{n-1}\frac{\rho^{2}}{(1+\rho)^{2}}\frac{2\beta}{n^{2}}(n-i)^{2}
\nonumber
\\
&\rightarrow\frac{2\beta}{3}\frac{\rho^{2}}{(1+\rho)^{2}},
\nonumber
\end{align}
as $n\rightarrow\infty$. 
\end{proof}

\begin{proposition}\label{CLTPropII}
Assume $b_{i}$ are i.i.d. non-negative random variables with finite mean.
\begin{equation}
\frac{\log x_{n}-n\log(1+\rho)}{\sqrt{n}}
\rightarrow N\left(0,\frac{2\beta}{3}\frac{\rho^{2}}{(1+\rho)^{2}}\right),
\end{equation}
in distribution as $n\rightarrow\infty$.
\end{proposition}

\begin{proof}
Note that
\begin{equation}
x_{n}=x_{0}\prod_{i=0}^{n-1}a_{i}+b_{0}\prod_{i=1}^{n-1}a_{i}+b_{1}\prod_{i=2}^{n-1}a_{i}
+\cdots+b_{n-2}a_{n-1}+b_{n-1}.
\end{equation}
Since $a_{i}\geq 1$ and $b_{i}\geq 0$, we have $x_{n}\geq x_{0}\prod_{i=0}^{n-1}a_{i}$
and
\begin{equation}
x_{n}\leq(x_{0}+b_{0}+b_{1}+\cdots+b_{n-1})\prod_{i=0}^{n-1}a_{i}
\end{equation}
By strong law of large numbers, $\frac{1}{n}(x_{0}+b_{0}+b_{1}+\cdots+b_{n-1})\rightarrow\mathbb{E}[b_{0}]$
a.s. as $n\rightarrow\infty$. Thus, $\frac{1}{\sqrt{n}}\log(x_{0}+b_{0}+b_{1}+\cdots+b_{n-1})\rightarrow 0$ a.s. as $n\rightarrow\infty$.
The result then follows from Proposition \ref{CLTProp}.
\end{proof}

%%%%%%%%  The Variational Problem %%%%%%%%%%%%%%%%%%%%%%%%%%%%%

\section{The Variational Problem}\label{VariationalSection}

We give in this Section the solution of the variational problem in Theorem~\ref{LyapunovThm}
for the Lyapunov exponent $\lambda(\rho,\beta)$. 
%We would like to solve the variational problem
%\begin{equation}\label{1}
%\lambda(\rho,\beta) = \mbox{sup}_{g(x)}
%\left\{ \log\rho g(1) + \beta \int_0^1 dx (g(1) - g(x))^2 - \int_0^1
%dx I(g'(x)) \right\}
%\end{equation}
%subject to the constraints on the function $g(x)$ 
%\begin{eqnarray}
%g(0) = 0\,,\quad
%0 \leq g'(x) \leq 1\,.
%\end{eqnarray}
The variational problem in (\ref{theorem1}) can be formulated equivalently 
in terms of the function $f(x)=g'(x)$
and the functional $\Lambda[f]$ defined as
\begin{eqnarray}\label{0}
&& \lambda(\rho,\beta) = \mbox{sup}_{f(x)} \Lambda[f]\\
&& \Lambda[f] \equiv
\log\rho \int_0^1 dx f(x) + 
\beta \int_0^1 dx \Big(\int_x^1 dy f(y)\Big)^2 
- \int_0^1 dx I(f(x)) \nonumber
\end{eqnarray}
defined in terms of a function $f:[0,1]\to [0,1]$
subject to the constraints
\begin{eqnarray}
0 \leq f(x) \leq 1\,.
\end{eqnarray}

%\subsection{Euler-Lagrange Equation}

The functional $\Lambda[f]$ defined in (\ref{0}) can be written in a more 
symmetrical form as
\begin{align}\label{symm}
\Lambda[f]&=\log\rho \int_0^1 dx f(x) + 
\beta \int_0^1 dx \Big(\int_x^1 dy f(y)\Big)^2 
- \int_0^1 dx I(f(x)) \\
&= \int_0^1 dx \Big\{
\log\rho f(x) + \beta \int_0^1 dy \int_0^1 dz K(z,y) f(z) f(y) - I(f(x))
\Big\}\,,\nonumber
\end{align}
where the kernel $K(z,y)$ is $K(z,y) = \min(z,y)$.

%\begin{proof}
%The relation (\ref{symm}) follows by writing
%\begin{align}
%\int_0^1 dx \Big(\int_x^1 dy f(y)\Big)^2 
%&= \int_0^1 dx \int_0^1 dy \int_0^1 dz
%\theta(y-x) \theta(z-x) f(y) f(z) 
%\\
%&=\int_0^1 dy \int_0^1 dz K(y,z) f(y) f(z),
%\nonumber
%\end{align}
%where $\theta(x)$ is the Heaviside step function
%\begin{align}
%\theta(x) = \left\{\begin{array}{cc}
%1\,, & x > 0 \\
%0\,, & x < 0 \\
%\end{array}
%\right.
%\end{align}
%The kernel is given by the integral
%\begin{eqnarray}
%K(y,z) = \int_0^1 dx \theta(y-x) \theta(z-x) = \mbox{min}(y,z)\,.
%\end{eqnarray}
%This concludes the proof of the relation  (\ref{symm}).
%\end{proof}

Taking the functional derivative of (\ref{symm}) with respect to $f$
we get the Euler-Lagrange equation
\begin{eqnarray}\label{EulerLagrange}
\frac{\delta \Lambda[f]}{\delta f} = 
\log\rho + 2\beta  \int_0^1 dz K(y,z) f(z) - \log\frac{f(y)}{1-f(y)} = 0\,.
\end{eqnarray}
%The optimizer function $f(y)$ is determined as the solution of this integral 
%equation.
This integral equation can be transformed into a differential equation by 
writing out the integral over the kernel in an explicit form
\begin{eqnarray}\label{LE2}
\log\rho + 2\beta  \int_0^y  z f(z) dz + 2\beta y \int_y^1   f(z) dz
- \log\frac{f(y)}{1-f(y)} = 0\,.
\end{eqnarray}
Take one derivative with respect to $y$
\begin{eqnarray}\label{ELp}
2\beta \int_y^1 dz  f(z) = \frac{d}{dy} \log \frac{f(y)}{1-f(y)}
= \frac{1}{f(y)(1-f(y))} f'(y) \,.
\end{eqnarray}

%Substituting into (\ref{LE2}), we get
%\begin{align}\label{LE3}
%&\log\rho + 2\beta  \int_0^y z f(z) dz + \frac{y}{f(y)(1-f(y))} f'(y) 
%- \log\frac{f(y)}{1-f(y)} \\
%&\qquad
%=\log\rho + 2\beta  \int_0^y z f(z) dz + y^2 \frac{d}{dy}
%\left( \frac{1}{y} \log\frac{f(y)}{1-f(y)} \right) = 0 \nonumber.
%\end{align}

Taking another derivative with respect to $y$ we obtain finally a second
order differential equation for the optimizer function $f(y)$
\begin{eqnarray}\label{LE4}
&& 2\beta  f(y)  = 
%-\frac{d}{dy} \Big\{ y^2 \frac{d}{dy}
%\Big( \frac{1}{y} \log\frac{f(y)}{1-f(y)} \Big) \Big\} =
- \frac{d^2}{dy^2} \log\frac{f(y)}{1-f(y)}\,.
\end{eqnarray}
This must be solved with the boundary conditions
\begin{eqnarray}\label{BC}
f(0) = \frac{\rho}{1+\rho}\,, \qquad f'(1) = 0\,.
\end{eqnarray}
The first boundary condition (at $y=0$) is obtained by taking $y=0$ in the 
Euler-Lagrange equation (\ref{EulerLagrange}). The integral vanishes and we 
get an equation for $f(0)$ which is solved with the result shown above.
The second boundary condition (at $y=1$) is obtained by taking $y=1$ in
equation (\ref{ELp}). 
%The specification of the boundary conditions is 
%somewhat unusual for an ordinary differential equation. They are specified at 
%different points, which makes the problem more similar to an eigenvalue problem
%than an initial value problem. 

%{\bf Some properties of the solution.} Some general properties of the
%solution for the optimizer $f(x)$ can be found even without a detailed
%numerical solution of the Euler-Lagrange equation.

%\begin{enumerate}

%\item The slope at origin $f'(0)$ is related to the zero-th moment of $f(x)$. 
%Taking $y=0$ in (\ref{ELp}) gives
%\begin{eqnarray}\label{18}
%f'(0) = 2\beta f(0)(1-f(0)) \int_0^1 dz f(z) = \frac{2\beta\rho}{(1+\rho)^2}
%\int_0^1 dz f(z)\,.
%\end{eqnarray}
%This is always positive. The inequality $f(z) < 1$ implies that the slope $f'(0)$
%is bounded from above as
%\begin{eqnarray}\label{fpbound}
%f'(0) \leq \frac{2\beta\rho}{(1+\rho)^2}\,.
%\end{eqnarray}

%\item The function $f(x)$ is everywhere non-decreasing $f'(x) \geq 0$. 
%This follows again from (\ref{ELp}) and the bounds $0 \leq f(x) \leq 1$.

\begin{remark}\label{remark6}
The functional $\Lambda[f]$ given in (\ref{symm}) has a simple physical 
interpretation: this is related to the  Landau potential 
(grand potential) $\Omega$ of a gas with
density $f(y)$ enclosed in a box $(0,1)$. The particles of the gas interact
by an attractive 2-body interaction with potential $-2K(z,y)$. The gas
is in contact with a thermostat of temperature $T=1/\beta$ and a reservoir of 
particles with chemical potential $\mu = T\log\rho$. 

This can be seen by writing the extremal value of the functional (\ref{symm}) as
\begin{eqnarray}
\lambda(\rho,\beta) = 
-\frac{1}{T} \Omega = -\frac{1}{T} (U - TS - \mu N)
\end{eqnarray}
with $N$ the total particle number, $S$ the entropy and $U$ the energy
\begin{eqnarray}
N &=& \int_0^1 dx f(x) \\
S &=& - \int_0^1 dx I(f(x)) \\
U &=& - \int_0^1 dy dz K(y,z) f(y) f(z) 
\end{eqnarray}
and $f(x)$ is given by the solution of the Euler-Lagrange (\ref{EulerLagrange}).
The equation for the density $f(x)$ is an analog of the isothermal Lane-Emden 
equation discussed in the canonical ensemble in \cite{MS}.
\end{remark}

\subsection{Solution of the Euler-Lagrange equation}

The equation (\ref{LE4}) can be alternatively expressed in terms of the function
\begin{eqnarray}
h(y) = \log\frac{f(y)}{1-f(y)}\,.
\end{eqnarray}
The function $h(y)$ satisfies
\begin{eqnarray}\label{heq}
h''(y) = - 2\beta \frac{e^{h(y)}}{1+e^{h(y)}} 
\end{eqnarray}
with boundary conditions
\begin{eqnarray}\label{hbc}
h(0) = \log\rho\,,\qquad
h'(1) = 0\,.
\end{eqnarray}

The solution of this equation is presented in the Appendix~\ref{app1}. 
Furthermore, it can be shown that the functional 
$\Lambda[h]$ depends only on $h(1)$, and is given explicitly by 
the following Proposition.

\begin{proposition}\label{propLambda}
If $f(x)$ is a solution of the Euler-Lagrange equation (\ref{EulerLagrange}), then 
$\Lambda[f]$ depends only on $f(1)$ or equivalently $h(1)$.
This is given by the relation
\begin{align}\label{Lambdah1}
\Lambda[f] &= \log(1 + e^{h(1)}) - \frac{1}{\sqrt{\beta}}
\int_{\log\rho}^{h(1)} dx \sqrt{\log\frac{1+e^{h(1)}}{1+e^x}}\,.
\end{align}
\end{proposition}

\begin{proof}
The proof is given in the Appendix~\ref{app1}.
\end{proof}

In conclusion, the Lyapunov exponent is given by
\begin{eqnarray}\label{var2}
\lambda(\rho,\beta) = \mbox{sup}_{h(1)} \Lambda[h(1)]\,.
\end{eqnarray}
%where the range of $h(1)$ is $h(1) \in\left[\log\rho, \log((1+\rho)e^\beta-1)\right]$.
The condition $\frac{d\Lambda[x]}{dx} = 0$ gives an equation for $h(1)$
\begin{eqnarray}\label{h1eq}
F(h(1);\rho) = 2\sqrt{\beta}\,
\end{eqnarray}
where $F(a;\rho)$ is given by
\begin{eqnarray}\label{Fdef}
F(a;\rho) \equiv \int_{\log\rho}^{a}
\frac{dx}{\sqrt{ \log\frac{1+e^a}{1+e^x}}}\,.
\end{eqnarray}

The variational problem is now solved in 2 steps.

1. For given $(\rho,\beta)$, find $h(1)$ by solving the equation (\ref{h1eq}).
This could have one or three solutions, depending on the values of $\rho,\beta$.

2. For each solution
for $h(1)$ obtained in the previous step, compute the functional $\Lambda[f]$
using equation (\ref{Lambdah1}). Take the supremum over these values.
This gives the Lyapunov exponent $\lambda(\rho,\beta)$.

From this solution one can see that the Lyapunov exponent 
$\lambda(\rho,\beta)$ is a continuous function
of its arguments, as it is the supremum of a family of continuous functions,
depending in a continuous way on the parameters $\rho,\beta$. 

%%%%%%%%%%%%%%%%%%%%%%%%%%%%%%%%%%%%%%%%%%%%%%%%%%%%%%%%%%%%%

%%%%%%%%%%%%%   Numerical solution   %%%%%%%%%%%%%%%%%%%%%%
\section{Numerical study of the solution}\label{NumericalSection}

We study in this section the numerical solution of the variational problem 
described in the previous section,
and present the results for the Lyapunov exponent.
The first step of the solution consists in finding $h(1)$ from the 
solution of the equation (\ref{h1eq}). We show in Figure~\ref{Fig:Fa} plots of 
the function $F(a;\rho)$ for several
values of $\rho= 0.01, 0.05,0.123, 0.2, 0.5$.
The intersection of
the curve $F(a;\rho)$ with the horizontal line $2\sqrt{\beta}$ determines 
$a = h(1)$ for given $\beta$.

We note from Figure~\ref{Fig:Fa} that the shape of $F(a;\rho)$ is qualitatively 
different for $\rho$  above or below a certain critical value $\rho_{c} \simeq 0.12$.
For $\rho > \rho_{c}$ the function $F(a;\rho)$ is strictly increasing, while for 
$\rho < \rho_{c}$ it has a minimum and a maximum. 
At $\rho = \rho_{c}$ the function $F(a;\rho_c)$ has an inflection point.
%In all cases the function
%$F(a;\rho)$ takes values in $(0,\infty)$ as $a\in (\log\rho, \infty)$, such that
%the equation (\ref{h1eq}) will always have a solution.
If $\rho > \rho_{c}$ the equation (\ref{h1eq}) has a unique
solution for $h(1)$, while for $\rho < \rho_{c}$ this equation has one or three
solutions, depending on $\beta$.

\begin{figure}[t]
\begin{center}
\includegraphics[scale=0.8]{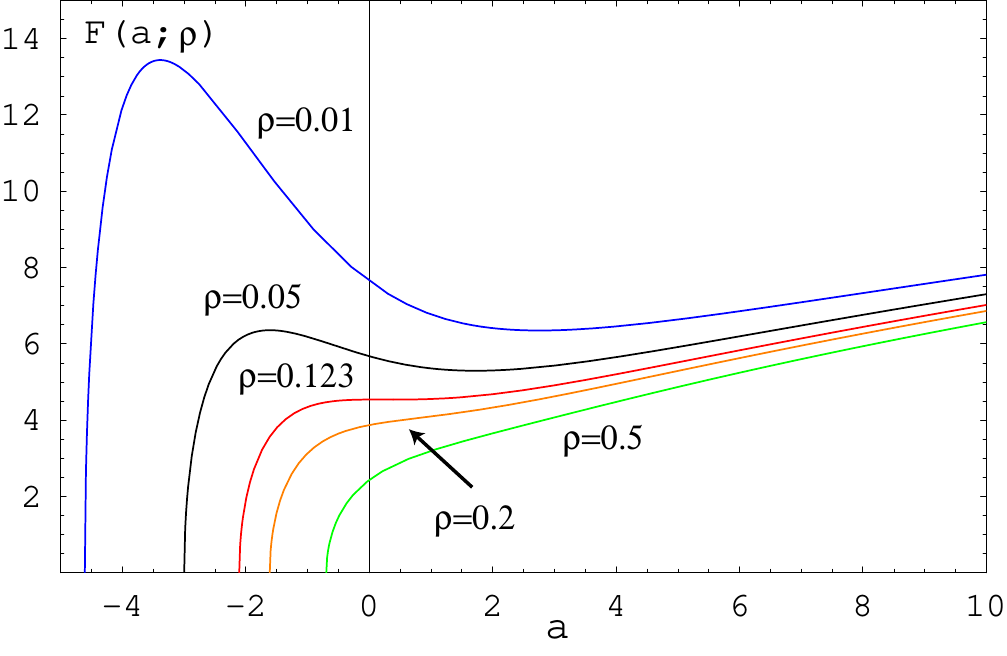}
\end{center}
\caption{
The function $F(a;\rho)$ vs $a$ for several values of $\rho$ as shown. 
}
\label{Fig:Fa}
\end{figure}

The Lyapunov exponent $\lambda(\rho,\beta)$ is given by the supremum of
the functional $\Lambda[h(1)]$ given in (\ref{Lambdah1})
over the solutions of the equation $F(a;\rho) = 2\sqrt{\beta}$.
The functional $\Lambda[h(1)]$ takes a simpler form if we choose as 
independent variable the integral $I_0$ instead of $h(1)$. 
These two variables are related by
\begin{equation}\label{dh1}
d \equiv \int_0^1 dx f(x) = \frac{(1+\rho)^2}{2\beta\rho} f'(0) = 
\frac{1}{2\beta}
h'(0) = \frac{1}{\sqrt{\beta}}
\sqrt{\log\frac{1+e^{h(1)}}{1+\rho}} \,.
\end{equation}
The variable $d \equiv I_0$ takes values in $d\in (0,1)$.

We have
\begin{eqnarray}\label{Lam2}
\Lambda(d) = \beta d^2 + \log(1+\rho) - 2\beta (1+\rho) d^3 \int_0^1
\frac{y^2 dy}{1+\rho - e^{\beta d^2 (y^2-1)}}\,.
\end{eqnarray}
We will use this expression for the numerical evaluation of the Lyapunov
exponent $\lambda(\rho,\beta) = \mbox{sup}_{d\in (0,1)} \Lambda(d)$.

%Typical results for $\Lambda(d)$ for sufficiently
%small $\rho < \rho_{c}$ are shown in Figure~\ref{Fig:Lam2} for $\rho = 0.025$
%(left plot) and $\rho = 0.115$ (right plot). This function has
%two maxima, which will be denoted as $d_{1,2}$ with $d_1 < d_2$. 
%As $\rho$ increases, the two maxima become closer, such that $d_2 - d_1$ 
%goes to zero as $\rho \to \rho_{c}$. For $\rho > \rho_{c}$ the function 
%$\Lambda(d)$ has a single maximum. 
%As $\beta$ crosses the transition value $\beta_{\rm cr}(\rho)$,
%the values $\Lambda(d_1)$ and $\Lambda(d_2)$ become equal and the supremum
%switches from one maximum to the other. At this point the Lyapunov exponent
%has a discontinuous derivative. For $\rho > \rho_{c}$ the
%Lyapunov exponent is a smooth function of $\beta$ as $\Lambda(d)$ has a single
%maximum $d_0$, and we have $\lambda(\rho,\beta) = \Lambda(d_0)$.

%\begin{figure}[t]
%\begin{center}
%\includegraphics[scale=0.6]{Lambdavsdrho025f.pdf}
%\includegraphics[scale=0.6]{Lambdavsdrho115f.pdf}
%\end{center}
%\caption{Plot of $\Lambda(d)$ vs $d$ at fixed $\rho$ for three values of $\beta$
%around the phase transition value $\beta_{\rm cr}$.
%Left plot: $\rho=0.025$ 
%and $\beta = 9.615$ (blue), 9.479 (black) and 9.346 (red). 
%Right plot: $\rho = 0.115$ and $\beta = 5.4$ (blue), 5.32 (black), 5.26 (red).
%}
%\label{Fig:Lam2}
%\end{figure}

We present in Figure \ref{Fig:Lyapunov} plots of the Lyapunov 
exponent $\lambda(\rho,\beta) = \mbox{sup}_d \Lambda(d)$ 
These plots show $\lambda(\rho,\beta)$ vs
$\beta$ at fixed $\rho=0.025,0.05,0.125,0.2$. 
The Lyapunov exponent has a discontinuous derivative at 
$\beta_{\rm cr}(\rho)$ for $\rho<\rho_{c}$.
The phase transition curve $\beta_{\rm cr}(\rho)$ is shown in Figure \ref{Fig:pd} as 
the black solid curve. It ends at the critical point $C$ with coordinates
\begin{eqnarray}
\rho_{c} = 0.12328\,,\quad
\beta_c = 5.12013\,,\quad
%T_c = 0.1953074\,,\quad
d_c = 0.372\,.
\end{eqnarray}

Next we study the dependence of the parameter $d$ on $\rho$ and $\beta$.
Start by keeping $\rho$ fixed and consider the dependence of $d$ on $\beta$. 
We have $\lim_{\beta \to \infty} d=1$, as
this value maximizes the second term in the functional $\Lambda[f]$. 
and $\lim_{\beta\to 0} d=\frac{\rho}{1+\rho}$, as this value
maximizes the sum of the first and last terms in  $\Lambda[f]$. 
For intermediate values of $\beta$, provided that $\rho < \rho_{c}$, 
$d$ has a jump discontinuity at a
certain value $\beta_{\rm cr}(\rho)$ and jumps from a value $d_2$ to $d_1$. 
If $\rho > \rho_{c}$, $d$ is a continuous function of $\beta$. 
For $\rho = \rho_{c}$, $d$ is continuous but its derivative with 
respect to $\beta$ becomes infinite at the $\beta = \beta_{c}$ point.
This behavior is shown in Fig.~\ref{Fig:a12}, which shows plots of 
$d = d(\rho,\beta)$ vs $1/\beta$ for several values of $\rho$.
The dashed curves show also $d_1,d_2$ vs $1/\beta$ along the phase
transition curve. 

A similar picture holds for the dependence of $d$ on $\rho$ at fixed $\beta$.
If $\beta < \beta_{c}$, $d$ is a continuous function of $\rho$. 
If $\beta > \beta_{c}$ it has a discontinuity at some value $\rho_{\rm cr}(\beta)$,
and if $\beta = \beta_{c}$ its derivative at $\rho = \rho_{c}$ is infinite. 

In Appendix \ref{app2} we derive an analytical approximation for the
functional $\Lambda(d)$ in the limit $\beta d^2 \gg 1$, which is used to 
obtain the properties of the phase transition curve and of $(d_1,d_2)$
along the phase transition curve in the $\beta \to \infty$ limit. 
This is summarized by the following result.

\begin{proposition}\label{largebeta}
The solutions of the variational problem for $\Lambda(d)$ given in 
Eq.~(\ref{Lam2}) approach the following limits for very large $\beta$ 
($\rho \ll 1$)
\begin{eqnarray}\label{d2largebeta}
\lim_{\beta\to\infty} d_1 = 0 \,,\qquad
\lim_{\beta\to \infty} d_2 = \frac34 \,.
\end{eqnarray}
The phase transition curve is given in the same limit by
\begin{eqnarray}\label{ptlimit}
\beta_{\rm cr}(\rho) = -\frac83 \log\Big(\frac{\rho}{1+\rho}\Big) 
\mbox{  as  } \rho \to 0\,.
\end{eqnarray}
\end{proposition}
See Appendix \ref{app2} for the proof. 
%The behavior of $d_1, d_2$ obtained from the numerical solution as
%$1/\beta \to 0$ shown in Figure~\ref{Fig:a12} agrees with these results. 

\begin{figure}[t]
\begin{center}
\includegraphics[scale=0.8]{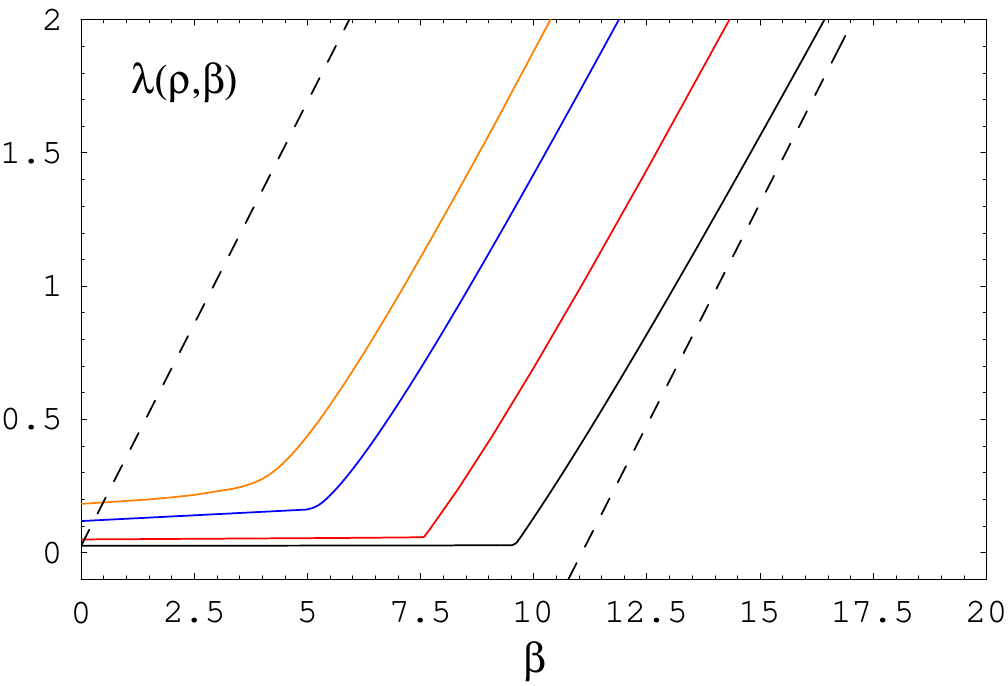}
\end{center}
\caption{
The Lyapunov exponent $\lambda(\rho,\beta)$ vs $\beta$ 
at fixed $\rho=0.025$ (black), $\rho = 0.05$ (red), $\rho = 0.125$ (blue)
and $\rho =0.2$ (orange).
The partial derivative $\partial_\beta\lambda(\rho,\beta)$ is
discontinuous at one point $\beta_{\rm cr}(\rho)$ for $\rho < \rho_{c}$. 
The dashed curves show the upper and lower bounds
on $\lambda(\rho,\beta)$ for $\rho = 0.025$ given in (\ref{bound1}). 
The upper bound is saturated at $\beta=0$ as shown in 
Proposition~\ref{prop:bounds} (i).
}
\label{Fig:Lyapunov}
 \end{figure}

We comment on the relation of these results to the statistical mechanical
interpretation of the variational problem.
As mentioned in the Remark~\ref{remark6}, the Lyapunov exponent 
$\lambda(\rho,\beta)$ is related to the equilibrium value of the
thermodynamical potential $\Omega(T,\mu)$ of a gas of particles interacting
by the two-body potential $V(x,y) = -2\mbox{min}(x,y)$. The gas is 
enclosed in a box $x\in (0,1)$ and is maintained at fixed temperature $T=1/\beta$ and
chemical potential $\mu = T\log\rho$, which corresponds to the grand canonical
ensemble. The solution of the variational problem discussed here demonstrates
the presence of a phase transition in this system. The equilibrium density of 
the gas is given by $f(x)$. 
%The density $f(x)$ approaches a uniform distribution in the very large temperature
%limit $\beta \to 0$. 
Both the distribution density $f(x)$ and its average value 
$d$ are discontinuous across the phase transition curve $\beta_{\rm cr}(\rho)$.

This system is similar to the lattice gas considered in \cite{DPII}. This
paper considered a lattice gas with $n$ sites with interaction energy 
$V(x,y) = -\frac{2}{n} K(x,y)$ (after appropriate rescaling of the lattice 
volume to 1).
The thermodynamical properties of this system have been computed in the 
thermodynamical limit $n\to \infty$ in the isobaric-isothermal ensemble. 
The results obtained in \cite{DPII} for the equation of state and the phase 
transition curve are in agreement with those obtained here directly in the 
grand canonical ensemble.

\begin{figure}[t]
\begin{center}
\includegraphics[scale=0.8]{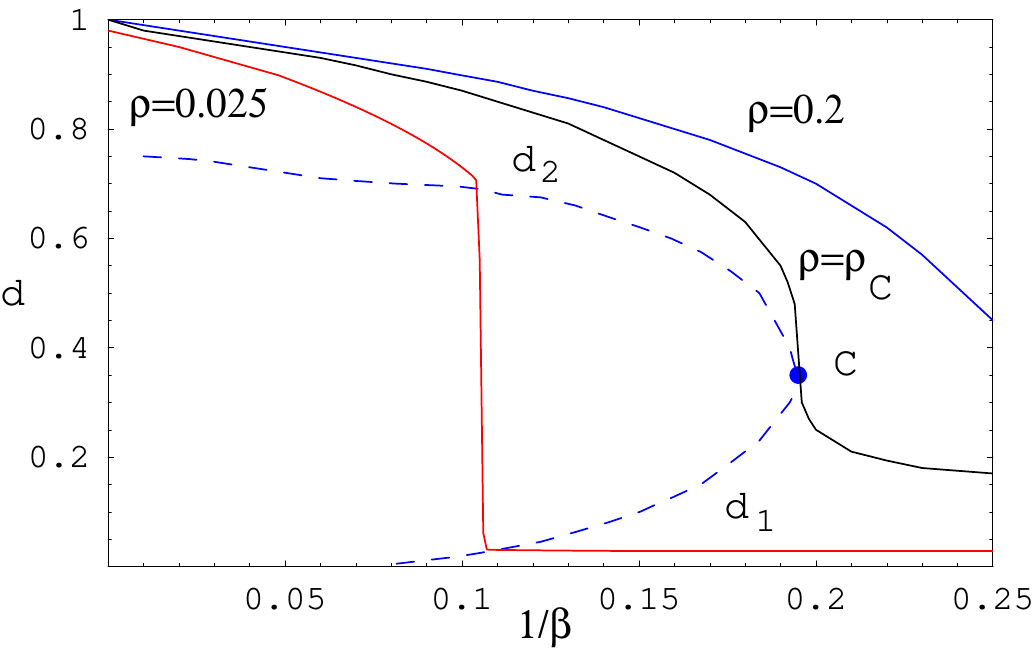}
\end{center}
\caption{
Plots of $d=d(\rho,\beta)$ vs $1/\beta$ at fixed $\rho=0.025$ (red),
$\rho = \rho_{c}$ (black) and $\rho= 0.2$ (blue).
The dashed blue curves show the values of the optimizer variable
$d_1,d_2$ vs $1/\beta$ along the phase transition curve. They 
become equal at the critical point $C$.
}
\label{Fig:a12}
\end{figure}

\section{Mean-Field Approximation}\label{MeanFieldSection}

A simple approximation to the solution of the variational problem is obtained by 
assuming that the function $f(x)\equiv a$ is a constant.
This corresponds to finding the supremum over $a \in (0,1)$ of the function
\begin{eqnarray}\label{meanfield}
\bar\lambda(\rho,\beta) = \mbox{sup}_{a\in (0,1)}
\left\{a\log\rho + \frac13 \beta a^2 - I(a)\right\}.
\end{eqnarray}
$\bar\lambda(\rho,\beta)$ gives a lower bound for the exact rate function in (\ref{0})
\begin{eqnarray}
\lambda(\rho,\beta) \geq \bar\lambda(\rho,\beta)\,.
\end{eqnarray}

This approximation leads to the well-known Curie-Weiss mean field theory 
\cite{Kadanoff,Stanley,Ellis}. 
This formula also appears as the limiting free energy of an edge and 2-star model
in the exponential random graph models whose properties have been studied rigorously
in \cite{Chatterjee,Radin,AristoffZhu}. 
For the problem considered here the lower bound $\bar\lambda(\rho,\beta)$ gives 
a van der Waals approximation for the Lyapunov exponent following from 
the equation of state of the equivalent lattice gas with temperature $1/\beta$ 
and fugacity $\rho$ \cite{DPII}. 

We summarize below without proof the main results of the mean field 
approximation.

\begin{proposition}\label{prop7}
i) The function $\bar\lambda(\rho,\beta)$ is given by 
\begin{eqnarray}
\bar\lambda(\rho,\beta) = - \frac13\beta a_*^2 - \log(1-a_*)
\end{eqnarray}
where $a_* = a_*(\rho,\beta)$ is given by
\begin{eqnarray}
a_*(\rho,\beta) = \left\{
\begin{array}{cc}
a_2(\rho,\beta) \,, & \log\rho > -\frac13\beta \\
a_1(\rho,\beta) \,, & \log\rho < -\frac13\beta \\
\end{array}
\right.
\end{eqnarray}
where $a_{1,2}$ are the smallest and largest solutions of the equation
\begin{eqnarray}
\log\rho =  -\frac23\beta a + \log\frac{a}{1-a}
\end{eqnarray}
if this equation has multiple solutions, or the unique solution of this equation
if it has only one solution.

ii) The function $\bar\lambda(\rho,\beta)$ has discontinuous partial derivatives 
with respect to its arguments at points on the phase transition curve
\begin{eqnarray}\label{T0rho}
\log\rho + \frac13\beta = 0\,,
\end{eqnarray}
with $\beta > \beta_{c} = 6$. This curve ends at the critical point $(\rho_{c},\beta_{c})=
(e^{-2}, 6)$. Across this curve, the solution $a_*$ jumps between $a_{*1,2}
= \frac12(1 \mp \Delta(\beta))$
where $\Delta(\beta)$ is the positive non-zero 
solution of the nonlinear equation 
$\Delta = \tanh\left(\frac{1}{6}\beta \Delta\right)$.
\end{proposition}

\begin{figure}[t]
\begin{center}
\includegraphics[scale=0.8]{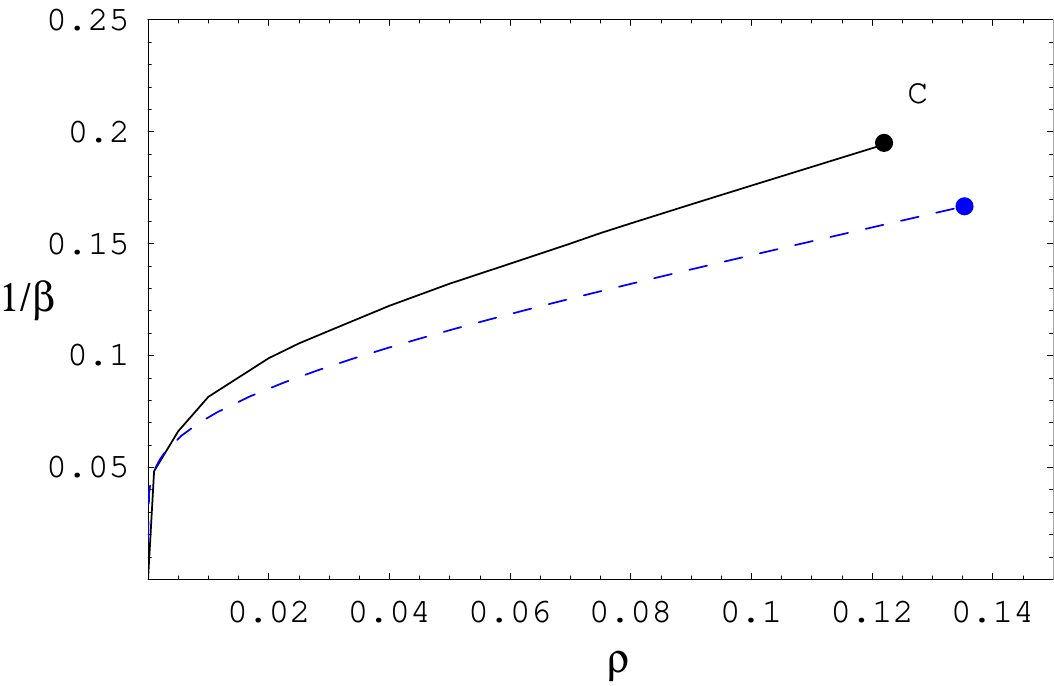}
\end{center}    
\caption{
The phase transition curve $\beta_{\rm cr}(\rho)$ in coordinates 
$(\rho, 1/\beta)$, obtained from the numerical solution of the model 
(black curve).
It ends at the critical point $C$. Blue dashed curve: 
the phase transition curve in the mean-field approximation,
which is given by Eq.~(\ref{T0rho}). 
}
\label{Fig:pd}
\end{figure}

\begin{remark}
The function $\bar\lambda(\rho,\beta)$ is related as 
$\bar\lambda(\rho,\beta) = \beta p_{\rm vdW}$ to the pressure of a van der Waals 
lattice gas with uniform long-range interaction
$\varepsilon_{ij} = - \frac{2}{3n}$ \cite{DPII}. This is identical to the 
equation of state of a lattice gas with Kac interaction \cite{LP}.
%This lower bound is saturated in the $\beta \to 0$ limit, when  the
%solution of the variational problem for $f(x)$ approaches a constant. 
\end{remark}

%\subsection{Critical Exponents in the Mean-Field Approximation}\label{CriticalMeanField}

The partial derivatives of $\bar\lambda(\rho,\beta)$ with respect to its
arguments are discontinuous across the phase transition curve. The discontinuity 
approaches zero near the critical point. 
The law of the discontinuity near  the critical point is
usually parameterized in terms of a critical exponent. Consider any 
quantity $M$ which is discontinuous across the phase transition curve. 
If the jump discontinuity approaches zero as $\beta \to \beta_{c}$ as
\begin{eqnarray}
\Delta M \simeq \gamma |\beta - \beta_{c}|^\alpha\,,
\end{eqnarray}
then $\alpha$ is called the \emph{critical exponent} of the quantity $M$.

The jump discontinuity of the solution $a$ of the
variational problem (\ref{meanfield}) for $G(a;\rho,\beta)$
is given by the following
well-known result of mean-field theory \cite{Kadanoff,Stanley,Ellis}. 

\begin{proposition}
Along the phase transition curve, very close to the critical point $\beta_{c}$,
the difference of the two solutions $a_1 \leq a_2$ of the variational problem 
(\ref{meanfield}) approaches zero as $\beta \to \beta_{c}$ as
\begin{eqnarray}
a_2(\beta) - a_1(\beta) = \Delta(\beta) \simeq
\sqrt{3} \left(\frac{\beta}{\beta_{c}}-1\right)^{\frac12}
\end{eqnarray}
\end{proposition}

Consider the jump of the partial derivatives of $\bar\lambda(\rho,\beta)$
with respect to $\rho$ and $\beta$ separately. Off the phase transition
curve we have
\begin{eqnarray}
\left(\frac{\partial \bar\lambda(\rho,\beta)}{\partial \rho}\right)_\beta = 
\frac{1}{\rho} a\,,
\end{eqnarray}
which means that this derivative has a jump discontinuity across the phase 
transition curve 
\begin{align}\label{part1}
\Delta \left(\frac{\partial \bar\lambda(\rho,\beta)}{\partial \rho}\right)_\beta = 
\frac{1}{\rho} (a_2(\beta) - a_1(\beta)) = \frac{1}{\rho} 
\Delta(\beta) \sim 
\frac{\sqrt{3}}{\rho} \left(\frac{\beta}{\beta_{c}}-1\right)^{1/2} \,.
\end{align}
%where we used in the last step the result of the Proposition~\ref{prop10}.

In a similar way, we have off the phase transition curve
\begin{align}
\left(\frac{\partial \bar\lambda(\rho,\beta)}{\partial \beta}\right)_\rho = 
\frac13 a^2\,,
\end{align}
which gives the following result for the jump of this partial derivative upon
crossing the phase transition curve
\begin{align}\label{part2}
\Delta \left(\frac{\partial \bar\lambda(\rho,\beta)}{\partial \beta}\right)_\rho = 
\frac13 (a_2^2(\beta) - a_1^2(\beta)) = \frac13 
\Delta(\beta) \sim 
\frac{1}{\sqrt{3}} \left(\frac{\beta}{\beta_{c}}-1\right)^{1/2}
\end{align}

%%%%%%%%%%%%%%%%  Phase transitions %%%%%%%%%%%%%%%%%%%%%%%%

\section{Phase Transitions}\label{PhaseSection}

In this section, we study rigorously the phase transitions for the Lypapunov exponent $\lambda(\rho,\beta)$.
Before we proceed, let us introduce the formal definition of phase transitions.
We will adopt the following definition of phase and phase transition  \cite{Ellis,Radin}.

\begin{definition}
A phase is a connected region of the parameter space $\{(\rho,\beta)\}$, maximal
for the condition that the Lyapunov exponent $\lambda(\rho,\beta)$ is analytic.
There is a $j$th-order phase transition at a boundary point of a phase if at least one $j$th-order
partial derivative of $\lambda(\rho,\beta)$ is discontinuous there, while
all lower order derivatives are continuous.
\end{definition}

\subsection{First-Order Phase Transition}

Recall that the Lyapunov exponent is given according to
equation (\ref{var2}) by
\begin{equation}
\lambda(\rho,\beta)=\log(1+e^{a})
-\frac{1}{\sqrt{\beta}}\int_{\log\rho}^{a}dx
\sqrt{\log\frac{1+e^{a}}{1+e^{x}}},
\end{equation}
where $a$ denotes the optimal value of $h(1)$ on which the
supremum of the functional $\Lambda[h(1)]$ is realized. 

Hence, if $a$ is unique, we can compute that
\begin{align}\label{dlambdadrho}
\frac{\partial\lambda}{\partial\rho}
&=\frac{e^{a}}{1+e^{a}}\left[1-\frac{1}{2\sqrt{\beta}}\int_{\log\rho}^{a}
\frac{dx}{\sqrt{\log\frac{1+e^{a}}{1+e^{x}}}}\right]\frac{\partial a}{\partial\rho}
+\frac{1}{\sqrt{\beta}}\frac{1}{\rho}\sqrt{\log\frac{1+e^{a}}{1+e^{\log\rho}}}
\\
&=\frac{1}{\sqrt{\beta}}\frac{1}{\rho}\sqrt{\log\left(\frac{1+e^{a}}{1+\rho}\right)}.
\nonumber
\end{align}
Similarly, we can compute that
\begin{align}\label{dlambdadbeta}
\frac{\partial\lambda}{\partial\beta}
=\frac{1}{2\beta^{3/2}}\int_{\log\rho}^{a}dx\sqrt{\log\frac{1+e^{a}}{1+e^{x}}}
=\frac{1}{2\beta}(-\lambda(\rho,\beta)+\log(1+e^{a}))\,.
\end{align}
Therefore, if $a$ is unique, the Lyapunov exponent is analytic and there is no 
phase transition. On the other hand, if there exist two distinct $a_{1}$ and 
$a_{2}$ that give the same value of $\lambda(\rho,\beta)$, then when we
change the parameters $(\rho,\beta)$ continuously, the value of $a$ changes and 
this causes a jump discontinuity of the first-order derivative, which leads to 
the first-order phase transition.
To summarize, if $F(a;\rho)=2\sqrt{\beta}$ has more than one solution that gives
the same value of Lyapunov exponent, then there is a first-order phase 
transition. Otherwise the Lyapunov exponent is analytic in its arguments.

\begin{proposition}
There exists $\underline{\rho}_{c}\in(0,\infty)$ so that for any $0<\rho<\underline{\rho}_{c}$, 
there exists a first-order phase transition.
\end{proposition}

\begin{proof}
Let us first analyze
\begin{equation}
F(a;\rho)=\int_{\log\rho}^{a}\frac{dx}{\sqrt{\log\frac{1+e^{a}}{1+e^{x}}}}.
\end{equation}
Changing the integration variable as $y=\log\frac{1+e^{a}}{1+e^{x}}$, we get
\begin{equation}
F(a;\rho)=\int_{0}^{\log\frac{1+e^{a}}{1+\rho}}\frac{1+e^{a}}{1+e^{a}-e^{y}}\frac{1}{\sqrt{y}}dy.
\end{equation}
Therefore, it suffices to study
\begin{equation}
G(a)=\int_{0}^{\log(\frac{a}{1+\rho})}\frac{a}{a-e^{y}}\frac{1}{\sqrt{y}}dy,
\qquad
a>1+\rho,
\end{equation}
which is related to $F(a;\rho)$  as $F(a;\rho)=G(1+e^{a})$.

It is straightforward to compute the derivative
\begin{equation}
G'(a)=\frac{1+\rho}{a\rho}\frac{1}{\sqrt{\log\frac{a}{1+\rho}}}
-\int_{0}^{\log(\frac{a}{1+\rho})}\frac{e^{y}}{(a-e^{y})^{2}}\frac{1}{\sqrt{y}}dy,
\end{equation}
from which it follows that one has $G'(1+\rho)=\infty$ and $G'(\infty)=0$.

Let $x=\frac{a}{1+\rho}$. Then, for any fixed $x>e$, using integration by parts,
\begin{align}
G'(a)&=\frac{1}{x\rho\sqrt{\log x}}
-\int_{0}^{\log x}\frac{e^{y}}{((1+\rho)x-e^{y})^{2}}\frac{1}{\sqrt{y}}dy
\\
&\leq
\frac{1}{x\rho\sqrt{\log x}}
-\int_{1}^{\log x}\frac{e^{y}}{((1+\rho)x-e^{y})^{2}}\frac{1}{\sqrt{y}}dy
\nonumber
\\
&=\frac{1}{x\rho\sqrt{\log x}}
-\int_{1}^{\log x}\frac{1}{\sqrt{y}}d\left(\frac{1}{(1+\rho)x-e^{y}}\right)
\nonumber
\\
&=\frac{1}{x\rho\sqrt{\log x}}-\frac{1}{x\rho\sqrt{\log x}}
+\frac{1}{(1+\rho)x-e}+\int_{1}^{\log x}\frac{1}{(1+\rho)x-e^{y}}d\left(\frac{1}{\sqrt{y}}\right)
\nonumber
\\
&=\frac{1}{(1+\rho)x-e}
-\frac{1}{2}\int_{1}^{\log x}\frac{1}{(1+\rho)x-e^{y}}\frac{1}{y^{3/2}}dy.
\nonumber
\end{align}
Observe that the first term in this expression is always positive
\begin{equation}
\lim_{\rho\rightarrow 0^{+}}\frac{1}{(1+\rho)x-e}=\frac{1}{x-e},
\end{equation}
while the second term diverges as $\rho \to 0_+$
\begin{align}
\frac{1}{2}\int_{1}^{\log x}\frac{1}{(1+\rho)x-e^{y}}\frac{1}{y^{3/2}}dy
&=\frac{1}{2}\int_{e}^{x}\frac{1}{(1+\rho)x-z}\frac{1}{(\log z)^{3/2}z}dz
\\
&\geq
\frac{1}{2(\log x)^{3/2}x}\int_{e}^{x}\frac{1}{(1+\rho)x-z}dz
\nonumber
\\
&=\frac{1}{2(\log x)^{3/2}x}\log\left(\frac{(1+\rho)x-e}{\rho x}\right)
\nonumber
\\
&\rightarrow+\infty,\nonumber
\end{align}
as $\rho\rightarrow 0^{+}$. Hence, we conclude that for any fixed 
$x=\frac{a}{1+\rho}>e$, the derivative $G'(a)$ becomes negative for
sufficiently small $\rho > 0$
\begin{equation}
\exists \underline\rho_x > 0\,, G'(a)<0 \mbox{ for all } 0 < \rho < 
\underline\rho_x\,.
\end{equation} 
%for sufficiently small $\rho>0$. 
Therefore, by continuity, 
for any $\rho$ sufficiently small, there exists an interval on which
$F'(a;\rho)$ is negative.

Moreover, for any $\rho > 0$, the function $F(a;\rho)$ grows without limit
as $a\to \infty$. This follows from the lower bound
\begin{align}
F(a;\rho)&=\int_{0}^{\log\frac{1+e^{a}}{1+\rho}}\frac{1+e^{a}}{1+e^{a}-e^{y}}
\frac{1}{\sqrt{y}}dy
\\
&\geq\int_{0}^{\log\frac{1+e^{a}}{1+\rho}}\frac{1}{\sqrt{y}}dy
=2\sqrt{\log\frac{1+e^{a}}{1+\rho}} \rightarrow\infty,
\nonumber
\end{align}
as $a\rightarrow\infty$. Hence, we conclude that $F(a;\rho)$ is increasing on 
an interval $(\log\rho,\log\rho+M)$
and $F(a;\rho)$ is decreasing on $(\log\rho+M,\log\rho+M+K)$ for some $M,K>0$ 
and $F(a;\rho)\rightarrow\infty$ as $a\rightarrow\infty$.
Hence, there is a first-order phase transition for any sufficiently small $\rho$.
\end{proof}

\begin{proposition}
There exists $\overline{\rho}_{c}\in(0,\infty)$, so that for any $\rho>\overline{\rho}_{c}$,
the Lyapunov exponent $\lambda(\rho,\beta)$ is analytic and hence there are no phase transitions.
\end{proposition}

\begin{proof}
Denoting $x = \frac{a}{1+\rho} \geq 1$ as above, we have
\begin{align}\label{Gp1bound}
G'(a)&=\frac{1}{x\rho\sqrt{\log x}}
-\int_{0}^{\log x}\frac{e^{y}}{((1+\rho)x-e^{y})^{2}}\frac{1}{\sqrt{y}}dy
\\
&\geq \frac{1}{x\rho \sqrt{\log x}}
-\frac{1}{\rho^{2}x^{2}}\int_{0}^{\log x}\frac{e^{y}}{\sqrt{y}}dy
\nonumber
\\
&=\frac{1}{\rho x\sqrt{\log x}} \left(1 - \frac{1}{\rho} H(\sqrt{\log x})\right)
\nonumber
\end{align}
where we denoted
\begin{eqnarray}
H(x) = xe^{-x^2} \int_0^{x^2} \frac{e^y}{\sqrt{y}} dy = 2 xe^{-x^2}
\int_0^x dt e^{t^2}\,.
% = 2\sqrt{\log x} D_+(\sqrt{\log x})
\end{eqnarray}

The function $H(x)$ has  the following properties: i) $H(0) = 0$; ii) $H(x)$ is 
positive for $x> 0$, and has a maximum at $x_0=1.502$ where it takes the value
$H(x_0)=1.28475$; iii) $\lim_{x\to \infty} H(x)=1$. The property iii) follows
by an application of the L'Hospital's rule
\begin{eqnarray}
\lim_{x\to \infty} \frac{2}{e^{x^2}\frac{1}{x}}\int_0^x dt e^{t^2} = 
\lim_{x\to \infty} \frac{2}{2-\frac{1}{x^2}} = 1\,.
\end{eqnarray}

These properties imply that the function $H(x)$ is bounded from above as
$H(x) \leq H(x_0)$, and thus for $\rho > H(x_0)$ the expression on the right-hand
side of (\ref{Gp1bound}) is positive for any $x>1$. 
This shows that for sufficiently large $\rho$, there can be only one optimal 
$h(1)$ and thus the Lyapunov exponent $\lambda(\rho,\beta)$ is analytic.
\end{proof}

\begin{remark}
Our analysis showed rigorously that there is no phase transition for sufficiently large $\rho$ and there is a first-order
phase transition for sufficiently small $\rho$ but we cannot show analytically what is in between. 
Numerical studies in Section \ref{NumericalSection} gave strong evidence that $\overline{\rho}_{c}=\underline{\rho}_{c}$.
\end{remark}

\subsection{Second-Order Phase Transition}

It has been shown in Eq.~(\ref{dlambdadrho}) that if the equation 
$F(a;\rho)=2\sqrt{\beta}$ has a unique solution for $a$, we have
\begin{equation}\label{firstderivative}
\frac{\partial\lambda}{\partial\rho}
=\frac{1}{\sqrt{\beta}}\frac{1}{\rho}\sqrt{\log\left(\frac{1+e^{a}}{1+\rho}\right)},
\end{equation}
where $a$ is the solution of the equation $F(a;\rho)=2\sqrt{\beta}$.

Let us define
\begin{equation}
\rho_{c}:=\sup\left\{\rho:\exists a,\frac{\partial}{\partial a}F(a;\rho)<0\right\}.
\end{equation}
From the proof that there is no first-order phase transition for large $\rho$
and there is first-order phase transition for small $\rho$, it is easy to see that $0<\rho_{c}<\infty$.
Moreover, for any $\rho>\rho_{c}$, $\frac{\partial}{\partial a}F(a;\rho)\geq 0$ for any $a$.
By continuity, $\frac{\partial}{\partial a}F(a;\rho_{c})\geq 0$ for any $a$.
It is also easy to see that there exists some $a_{c}$ so that $\frac{\partial}{\partial a}F(a_{c};\rho_{c})=0$ 
and define $\beta_{c}$ as $2\sqrt{\beta_{c}}=F(a_{c};\rho_{c})$. (Note that our analysis does not show that 
$\beta_{c}$ is unique although numerical results suggest so.)

\begin{proposition}\label{2ndorderpt}
There is a second-order phase transition at $(\rho_{c},\beta_{c})$. 
At this point all second order partial derivatives of $\lambda(\rho,\beta)$ are
infinite.
\end{proposition}

\begin{proof}
First, since $\frac{\partial}{\partial a}F(a;\rho_{c})\geq 0$ for any $a$, by \eqref{firstderivative}, there is no
first-order phase transition. Next, we can compute that
\begin{equation}\label{SecondRho}
\frac{\partial^{2}\lambda}{\partial\rho^{2}}
=\frac{1}{\sqrt{\beta}\rho\sqrt{\log\frac{1+e^{a}}{1+\rho}}}
\left[-\frac{1}{\rho}\log\left(\frac{1+e^{a}}{1+\rho}\right)
-\frac{1}{2(1+\rho)}+\frac{1}{2}\frac{e^{a}}{1+e^{a}}\frac{\partial a}{\partial\rho}\right],
\end{equation}
where $F(a;\rho)=2\sqrt{\beta}$. Off the phase transition curve, we differentiate
the equation $F(a;\rho)=2\sqrt{\beta}$ w.r.t. $\rho$ at fixed $\beta$. 
Using (\ref{dlambdadrho}) we have 
\begin{equation}
\frac{\partial}{\partial a}F(a;\rho)\frac{\partial a}{\partial\rho}=\frac{1}{\rho}\frac{1}{\sqrt{\log\frac{1+e^{a}}{1+\rho}}}.
\end{equation}
Since $\lim_{(\rho,\beta)\rightarrow(\rho_{c},\beta_{c})}\frac{\partial}{\partial a}F(a;\rho)=0$,
we have $\lim_{(\rho,\beta)\rightarrow(\rho_{c},\beta_{c})}\frac{\partial a}{\partial\rho}=\infty$.
Moreover, since $\frac{\partial}{\partial a}F(\log\rho;\rho)=\infty$, it is clear
that at $(\rho_{c},\beta_{c})$, $\log(\frac{1+e^{a}}{1+\rho})$ does not vanish.
Together, we proved that
\begin{equation}
\lim_{(\rho,\beta)\rightarrow(\rho_{c},\beta_{c})}\frac{\partial^{2}\lambda}{\partial\rho^{2}}=\infty.
\end{equation}

Consider next the partial derivative $\partial_\beta\lambda$.
This is given by
\begin{align}
\frac{\partial\lambda}{\partial\beta}
=\frac{1}{2\beta^{3/2}}\int_{\log\rho}^{a}dx\sqrt{\log\frac{1+e^{a}}{1+e^{x}}}
=\frac{1}{2\beta}(-\lambda(\rho,\beta)+\log(1+e^{a}))\,.
\end{align}
Since $\lambda(\rho,\beta)$ is continuous in its arguments, there is a first-order phase transition
if $a$ has a jump discontinuity, i.e. switching from one solution to the other 
of the equation $F(a;\rho)=2\sqrt{\beta}$.
It is clear that there is no first-order phase transition at $(\rho_{c},\beta_{c})$.
Since
\begin{align}\label{SecondBeta}
&\frac{\partial^{2}\lambda}{\partial\beta^{2}}
=-\frac{1}{2\beta^{2}}(-\lambda(\rho,\beta)+\log(1+e^{a}))
+\frac{1}{2\beta}\left(-\frac{\partial\lambda}{\partial\beta}
+\frac{e^{a}}{1+e^{a}}\frac{\partial a}{\partial\beta}\right)
\\
&=\frac{3}{4\beta^{2}}(\lambda(\rho,\beta)-\log(1+e^{a}))
+\frac{1}{2\beta}\frac{e^{a}}{1+e^{a}}\frac{\partial a}{\partial\beta},
\nonumber
\end{align}
and $\frac{\partial a}{\partial\beta}+\frac{1}{\sqrt{\beta}}\frac{1}{\frac{\partial}{\partial a}F(a;\rho)}\rightarrow\infty$
as $(\rho,\beta)\rightarrow(\rho_{c},\beta_{c})$,
there is a second-order phase transition at $(\rho_{c},\beta_{c})$.

Finally we consider the cross second derivative. This can be computed in
two ways, taking the derivatives in either order. We get
\begin{eqnarray}\label{SecondRhoBeta}
&& \frac{\partial^{2}\lambda}{\partial\rho\partial\beta}
=-\frac{1}{2\beta^{3/2}}\frac{1}{\rho}\sqrt{\log\frac{1+e^{a}}{1+\rho}}
+\frac{1}{\sqrt{\beta}}\frac{1}{\rho}\frac{1}{2\sqrt{\log\frac{1+e^{a}}{1+\rho}}}
\frac{e^{a}}{1+e^{a}}\frac{\partial a}{\partial\beta} \\
&& \qquad =
-\frac{1}{2\beta^{3/2}}\frac{1}{\rho}\sqrt{\log\frac{1+e^{a}}{1+\rho}}
+ \frac{1}{2\beta}
\frac{e^{a}}{1+e^{a}}\frac{\partial a}{\partial\rho} \,. \nonumber
\end{eqnarray}
Comparing them we get the following relation between the partial derivatives 
of $a$, the unique solution of the variational problem,
\begin{eqnarray}\label{Maxwell}
\frac{\partial a}{\partial\beta} = \frac{1}{\sqrt{\beta}}\rho
\sqrt{\log\frac{1+e^{a}}{1+\rho}}
\frac{\partial a}{\partial\rho}\,.
\end{eqnarray}
This relation holds only off the phase transition curve, or at the critical
point. Elsewhere on the phase transition curve, $a$ is discontinuous and its
partial derivatives do not exist. 
This is an analog of the Maxwell relations, which are well-known in 
thermodynamics \cite{Stanley}. 

We have seen above that at the critical point $(\rho,\beta)\rightarrow(\rho_{c},\beta_{c})$
both partial derivatives $(\partial a/\partial\rho)$ and $(\partial a/\partial\beta)$ 
become infinite. 
This implies that the limit of (\ref{SecondRhoBeta}) as
$(\rho,\beta)\rightarrow(\rho_{c},\beta_{c})$ is infinite also. 
\end{proof}

The divergence of the partial derivative $(\partial a/\partial\beta)_{\rho=\rho_c}$ 
at $\beta \to \beta_c$  is seen in graphical form in Figure~\ref{Fig:a12}. The black solid
curve shows $d(\rho_c,\beta)$ as function of $1/\beta$. Recalling that $d$ is related
to $a$ as (\ref{dh1}), one obtains 
\begin{eqnarray}
\frac{\partial d}{\partial\beta} = - \frac{1}{2\beta} d + \frac{1}{2\beta d} 
\left( 1- \frac{e^{-\beta d^2}}{1+\rho}\right) 
\frac{\partial a}{\partial\beta} \,.
\end{eqnarray}
This becomes infinite as we approach the critical point $\beta \to \beta_c$ along the 
curve of fixed $\rho = \rho_c$, as seen in Figure~\ref{Fig:a12}.

\section{The Slope of the Phase Transition Curve}
\label{Sec:slope}

We prove in this section a relation for the slope of the phase transition curve
$\beta_{\rm cr}(\rho)$. This is given by the following result.

\begin{theorem} (Clausius-Clapeyron relation)
The slope of the phase transition curve $\beta(\rho)$ is related to the ratio of 
the jump discontinuities of the first partial derivatives of $\lambda(\rho,\beta)$
with respect to its arguments. This ratio is given by
\begin{eqnarray}\label{CC}
\frac{d\beta_{\rm cr}}{d\rho} = 
- \frac{\Delta \Big(\frac{\partial \lambda(\rho,\beta)}{\partial \rho}\Big)}
{\Delta \Big(\frac{\partial \lambda(\rho,\beta)}{\partial \beta}\Big)} = 
- \frac{2}{\rho(d_1+d_2)} \,,
\end{eqnarray}
where the jump discontinuities of the derivatives across the 
phase transition curve are defined as
\begin{eqnarray}
\Delta \Big(\frac{\partial \lambda(\rho,\beta)}{\partial \rho}\Big) =
\Big(\frac{\partial \lambda(\rho,\beta)}{\partial \rho}\Big)|_{d_2} -
\Big(\frac{\partial \lambda(\rho,\beta)}{\partial \rho}\Big)|_{d_1}\,,
\end{eqnarray}
and analogous for $\Delta \Big(\frac{\partial \lambda(\rho,\beta)}{\partial \beta}\Big)$.
$d_1 < d_2$ are the solutions of the variational problem for $\Lambda(d)$ on the
phase transition curve. Recall that $d$ is related to $h(1)$ as shown in equation (\ref{dh1}).
The two solutions $d_{1,2}$ become equal at the critical point $\lim_{\rho \to\rho_{c}}
(d_2 - d_1) = 0$.
\end{theorem}

\begin{proof}
The proof of the relation (\ref{CC}) uses the continuity of the Lyapunov exponent
$\lambda(\rho,\beta)$ across the phase transition line. 
Equating the change in $\lambda(\rho,\beta)$ as we 
move along the transition curve, on one side and on the other side of the curve
respectively, we get
\begin{eqnarray}
&& d\lambda(\rho,\beta) = 
\Big(\frac{\partial \lambda(\rho,\beta)}{\partial \rho}\Big) d\rho +
\Big(\frac{\partial \lambda(\rho,\beta)}{\partial \beta}\Big) d\beta |_{\rm phase 1} \\
&& = \Big(\frac{\partial \lambda(\rho,\beta)}{\partial \rho}\Big) d\rho +
\Big(\frac{\partial \lambda(\rho,\beta)}{\partial \beta}\Big) d\beta |_{\rm phase 2} \,.
\nonumber
\end{eqnarray}
This gives immediately the first equality in equation~(\ref{CC}).

In order to prove also the second equality in (\ref{CC}), we use the explicit expressions
for the partial derivatives obtained in (\ref{dlambdadrho}) and (\ref{dlambdadbeta}).
These relations simplify when expressed in terms of the $d$ variable, related to $h(1)$
as in (\ref{dh1}). We obtain, off the phase transition curve,
\begin{eqnarray}
&& \Big(\frac{d\lambda}{d\rho}\Big) = \frac{d}{\rho} \\
&& \Big(\frac{d\lambda}{d\beta}\Big) = \frac{1}{2\beta}
\Big[ \beta d^2 + \log(1+\rho) - \lambda(\rho,\beta) \Big] \,.
\end{eqnarray}
The jump discontinuities of these derivatives are given by
\begin{eqnarray}
\Delta \Big(\frac{\partial \lambda(\rho,\beta)}{\partial \rho}\Big) = \frac{1}{\rho}(d_2-d_1) \\
\Delta \Big(\frac{\partial \lambda(\rho,\beta)}{\partial \beta}\Big) = \frac12 (d_2^2-d_1^2)
\end{eqnarray}
Using these expressions into (\ref{CC}) one finds the explicit result for the 
slope of the phase transition curve given in the second equality of (\ref{CC}).
\end{proof}

\begin{remark}
We note that the results of the Proposition~\ref{largebeta} are in 
agreement with the relation (\ref{CC}) for the slope of the phase transition 
curve. 
\end{remark}

\begin{remark}
A relation of the form (\ref{CC}) holds also in the mean-field approximation. For this
case we have to replace $d$ with $a$, which is the optimizer of the variational problem
for $G(a;\rho,\beta)$. The jump discontinuities of the first derivatives of 
$\bar\lambda(\rho,\beta)$ have been computed in (\ref{part1}) and (\ref{part2}), 
respectively, which gives the slope of the phase transition curve
\begin{eqnarray}
\frac{d\beta}{d\rho} = - \frac{\frac{1}{\rho} \Delta(\beta)}{\frac13 \Delta(\beta)}  = 
- \frac{3}{\rho}\,.
\end{eqnarray}
This agrees with the known result for the phase transition curve (\ref{T0rho})
in the mean-field approximation.
\end{remark}

A relation of the form (\ref{CC}) has been proved in \cite{AristoffZhu} for 
the slope of the phase transition curve in the $p-$star
ERGM, see Theorem 3 in \cite{AristoffZhu}. 

%%%%%%%%%%%%%%%%%%%%%%%%%%%%%%%%%%%%%%%%%%%%%%%%%%%%%%%%%%%%%

\section{Critical Exponent}\label{CriticalSection}

We have already seen that at $(\rho_{c},\beta_{c})$, 
$\frac{\partial}{\partial a}F(a_{c};\rho_{c})=0$ and $F$ is increasing everywhere.
Therefore, $\frac{\partial^{2}}{\partial a^{2}}F(a_{c};\rho_{c})=0$. (Otherwise, 
$a_{c}$ is a local minimum (resp. local maximum) if $\frac{\partial^{2}}{\partial a^{2}}F(a_{c};\rho_{c})>0$
(resp. $\frac{\partial^{2}}{\partial a^{2}}F(a_{c};\rho_{c})<0$), which contradicts $F$ being increasing
everywhere.)
Moreover, since $\frac{\partial}{\partial a}F(a;\rho_{c})>0$ for any $a$ in a neighborhood
of $a_{c}$ except at $a_{c}$ and $\frac{\partial}{\partial a}F(a_{c};\rho_{c})=0$,
we conclude that $a_{c}$ is a local minimum of the function $\frac{\partial}{\partial a}F(a;\rho_{c})$, 
which implies that $\frac{\partial^{3}}{\partial a^{3}}F(a_{c};\rho_{c})>0$.

Along the phase transition curve, there exist $a_{1}<a_{c}<a_{2}$ such that
$F(a_{1};\rho)=F(a_{2};\rho)$. If there exists the relation
\begin{equation}
a_{2}-a_{1}\simeq\gamma|\beta-\beta_{c}|^{\alpha},
\end{equation}
along the phase transition curve as $\beta\rightarrow\beta_{c}$, then, the exponent $\alpha$
is called the \emph{critical exponent} in statistical mechanics.

As a first step, we need to understand the asymptotic relation between 
$\rho-\rho_{c}$ and $\beta-\beta_{c}$ near the critical point 
$(\rho_{c},\beta_{c})$.

Along the phase transition curve,
\begin{align}\label{RewriteI}
&\log(1+e^{a_{1}})-\frac{1}{\sqrt{\beta}}\int_{\log\rho}^{a_{1}}\sqrt{\log\frac{1+e^{a_{1}}}{1+e^{x}}}dx
\\
&=\log(1+e^{a_{2}})-\frac{1}{\sqrt{\beta}}\int_{\log\rho}^{a_{2}}\sqrt{\log\frac{1+e^{a_{2}}}{1+e^{x}}}dx.
\nonumber
\end{align}
Differentiating with respect to $\rho$ and using the identity 
$F(a_{1};\rho)=F(a_{2};\rho)=2\sqrt{\beta}$, we get
\begin{align}\label{RewriteII}
&\frac{1}{\sqrt{\beta}}\frac{1}{\rho}\sqrt{\log\frac{1+e^{a_{1}}}{1+\rho}}
+\frac{1}{2}\frac{1}{\beta^{3/2}}\int_{\log\rho}^{a_{1}}\sqrt{\log\frac{1+e^{a_{1}}}{1+e^{x}}}dx\frac{\partial\beta}{\partial\rho}
\\
&=\frac{1}{\sqrt{\beta}}\frac{1}{\rho}\sqrt{\log\frac{1+e^{a_{2}}}{1+\rho}}
+\frac{1}{2}\frac{1}{\beta^{3/2}}\int_{\log\rho}^{a_{2}}\sqrt{\log\frac{1+e^{a_{2}}}{1+e^{x}}}dx\frac{\partial\beta}{\partial\rho}.
\nonumber
\end{align}
By \eqref{RewriteI}, we can rewrite \eqref{RewriteII} as
\begin{align}
&\frac{1}{\sqrt{\beta}}\frac{1}{\rho}\sqrt{\log\frac{1+e^{a_{1}}}{1+\rho}}
+\frac{1}{2}\frac{1}{\beta}\log(1+e^{a_{1}})\frac{\partial\beta}{\partial\rho}
\\
&=\frac{1}{\sqrt{\beta}}\frac{1}{\rho}\sqrt{\log\frac{1+e^{a_{2}}}{1+\rho}}
+\frac{1}{2}\frac{1}{\beta}\log(1+e^{a_{2}})\frac{\partial\beta}{\partial\rho},
\nonumber
\end{align}
which implies that
\begin{align}
\frac{\partial\beta}{\partial\rho}
&=-\frac{\frac{1}{\sqrt{\beta}}\frac{1}{\rho}\sqrt{\log\frac{1+e^{a_{2}}}{1+\rho}}
-\frac{1}{\sqrt{\beta}}\frac{1}{\rho}\sqrt{\log\frac{1+e^{a_{1}}}{1+\rho}}}{\frac{1}{2}\frac{1}{\beta}\log(1+e^{a_{2}})
-\frac{1}{2}\frac{1}{\beta}\log(1+e^{a_{1}})}
\\
&\rightarrow
-\frac{\sqrt{\beta_{c}}}{\rho_{c}}\frac{1}{\sqrt{\log\frac{1+e^{a_{c}}}{1+\rho_{c}}}},
\nonumber
\end{align}
as $a_{2}-a_{1}\rightarrow 0$ (and thus $a_{1},a_{2}\rightarrow a_{c}$, $(\rho,\beta)\rightarrow(\rho_{c},\beta_{c})$). 

Therefore, along the transition curve near the critical point, 
\begin{equation}\label{BetaRhoRelation}
\beta-\beta_{c}=-\frac{\sqrt{\beta_{c}}}{\rho_{c}}\frac{1}{\sqrt{\log\frac{1+e^{a_{c}}}{1+\rho_{c}}}}(\rho-\rho_{c})
+O((\rho-\rho_{c})^{2}).
\end{equation}
This agrees with the result of (\ref{CC}) which gives that at the critical point,
the slope of the phase transition curve is
\begin{eqnarray}
\frac{d\beta}{d\rho}\Big|_{\rho=\rho_{c}} = - \frac{1}{\rho_{c} d_c} = 
- \frac{\sqrt{\beta_{c}}}{\rho_{c}} \frac{1}{\sqrt{\log\frac{1+e^{a_c}}{1+\rho}}}\,.
\end{eqnarray}

Along the phase transition curve, we have
\begin{equation}
F(a_{1};\rho)-F(a_{c};\rho_{c})
=2\sqrt{\beta}-2\sqrt{\beta_{c}}.
\end{equation}

On the one hand,
\begin{align}\label{OneHand}
2\sqrt{\beta}-2\sqrt{\beta_{c}}&=\frac{1}{\sqrt{\beta_{c}}}(\beta-\beta_{c})+O((\beta-\beta_{c})^{2})
\\
&=-\frac{1}{\rho_{c}}\frac{1}{\sqrt{\log\frac{1+e^{a_{c}}}{1+\rho_{c}}}}(\rho-\rho_{c})+O((\rho-\rho_{c})^{2}).
\nonumber
\end{align}
On the other hand,
\begin{align}\label{OtherHand}
&F(a_{1};\rho)-F(a_{c};\rho_{c})
\\
&=[F(a_{1};\rho)-F(a_1;\rho_{c})]+[F(a_{1};\rho_{c})-F(a_{c};\rho_{c})]
\nonumber
\\
&=-\frac{1}{\rho_{c}}\frac{1}{\sqrt{\log\frac{1+e^{a_{1}}}{1+\rho_{c}}}}
(\rho-\rho_{c})+O((\rho-\rho_{c})^{2})
\nonumber
\\
&\qquad\qquad\qquad\qquad
+\frac{1}{6}\frac{\partial^{3}}{\partial a^{3}}F(a_{c};\rho_{c})(a_{1}-a_{c})^{3}+O((a_{1}-a_{c})^{4})
\nonumber
\end{align}
Therefore, by \eqref{OneHand} and \eqref{OtherHand},
\begin{align}
&-\frac{1}{\rho_{c}}\frac{1}{\sqrt{\log\frac{1+e^{a_{c}}}{1+\rho_{c}}}}(\rho-\rho_{c})+O((\rho-\rho_{c})^{2})
\\
&=-\frac{1}{\rho_{c}}\frac{1}{\sqrt{\log\frac{1+e^{a_{1}}}{1+\rho_{c}}}}(\rho-\rho_{c})+O((\rho-\rho_{c})^{2})
\nonumber
\\
&\qquad\qquad\qquad\qquad
+\frac{1}{6}\frac{\partial^{3}}{\partial a^{3}}F(a_{c};\rho_{c})(a_{1}-a_{c})^{3}+O((a_{1}-a_{c})^{4}),
\nonumber
\end{align}
which implies that
\begin{align}
&-\frac{1}{\rho_{c}}(\rho-\rho_{c})(a_{1}-a_{c})\frac{1}{2\left(\log\frac{1+e^{a_{c}}}{1+\rho_{c}}\right)^{3/2}}\frac{e^{a_{c}}}{1+e^{a_{c}}}
\\
&=\frac{1}{6}\frac{\partial^{3}}{\partial a^{3}}F(a_{c};\rho_{c})(a_{1}-a_{c})^{3}
\nonumber
\\
&\qquad\qquad
+O((a-a_{c})^{4})
+O((\rho-\rho_{c})^{2})+O(|\rho-\rho_{c}||a_{1}-a_{c}|).
\nonumber
\end{align}
Therefore, along the phase transition curve near the critical point, we have
\begin{equation}
a_{1}-a_{c} \simeq - D_c |\rho-\rho_{c}|^{1/2}
\end{equation}
with
\begin{equation}
D_c = 
\left(\frac{1}{\left(\log\frac{1+e^{a_{c}}}{1+\rho_{c}}\right)^{3/2}}
\frac{3e^{a_{c}}}{\rho_{c}(1+e^{a_{c}})}
\frac{1}{\frac{\partial^{3}}{\partial a^{3}}F(a_{c};\rho_{c})}\right)^{1/2}\,.
\end{equation}
Similarly,
\begin{equation}
a_{2}-a_{c}\simeq D_c |\rho-\rho_{c}|^{1/2}\,.
%\left(\frac{1}{\left(\log\frac{1+e^{a_{c}}}{1+\rho_{c}}\right)^{3/2}}\frac{3e^{a_{c}}}{\rho_{c}(1+e^{a_{c}})}
%\frac{1}{\frac{\partial^{3}}{\partial a^{3}}F(a_{c};\rho_{c})}\right)^{1/2}
\end{equation}
Hence, we conclude that
\begin{equation}
a_{2}-a_{1}\simeq 
2 D_c |\rho-\rho_{c}|^{1/2}.
\end{equation}

By \eqref{BetaRhoRelation}, we also have
\begin{equation}\label{a12critical}
a_{2}-a_{1}\simeq 
2 D_c \left(
\frac{\rho_c}{\sqrt{\beta_c}} \sqrt{\log\frac{1+e^{a_c}}{1+\rho_c}}\right)^{1/2}
%2\left(\frac{1}{\left(\log\frac{1+e^{a_{c}}}{1+\rho_{c}}\right)}\frac{3e^{a_{c}}}{\sqrt{\beta_{c}}(1+e^{a_{c}})}
%\frac{1}{\frac{\partial^{3}}{\partial a^{3}}F(a_{c};\rho_{c})}\right)^{1/2}
|\beta-\beta_{c}|^{1/2}.
\end{equation}
Therefore, the critical exponent is $1/2$.

\begin{remark}
We have showed that the critical exponent is $1/2$ which is the same as in 
the mean field approximation.
\end{remark}

This result can be used to obtain the law of the approach to zero of the
jump in the partial derivatives of the Lyapunov exponent near the critical
point. 

\begin{proposition} 
Near the critical point $\beta\to \beta_c, \rho\to \rho_c$, the jump
discontinuities of the partial derivatives of the Lyapunov exponent approach
zero as
\begin{eqnarray}
&& \Delta\Big( \frac{\partial \lambda}{\partial\beta} \Big) = \frac{1}{2\beta}
\log\frac{1+a_2}{1+a_1} \to c_1 (\beta - \beta_c)^{1/2} \\
&& \Delta\Big( \frac{\partial \lambda}{\partial\rho} \Big) 
\to c_2
(\beta - \beta_c)^{1/2} \,.
\end{eqnarray}
where $c_1, c_2$ are positive real constants given by
\begin{eqnarray}
c_1 &=& \frac{2D_c}{\beta_c(1+a_c)} \left(
\frac{\rho_c}{\sqrt{\beta_c}} \sqrt{\log\frac{1+e^{a_c}}{1+\rho_c}}\right)^{1/2} \\
c_2 &=& \frac{2D_c}{\beta_c(1+a_c)} \left(
\frac{\rho_c}{\sqrt{\beta_c}} \sqrt{\log\frac{1+e^{a_c}}{1+\rho_c}}\right)^{-1/2}\,.
\end{eqnarray}
\end{proposition}

\begin{proof}
Follows immediately from the relations (\ref{dlambdadbeta}) and (\ref{dlambdadrho})
for the partial derivatives, together with the result (\ref{a12critical}) for 
the jump of the optimizer variables along the critical line $a_2 - a_1$ near the
critical point $\beta \to \beta_c$.
\end{proof}

\section{Generalizations}
\label{Sec:10}

We consider in this Section two extensions of the results presented above:
a result for the growth rate of the positive integer
moments of the random variable $x_n$, and a generalization to the linear 
random recursion $x_{i+1} = a_i x_i + b_i$ with additive i.i.d. noise $b_i$.

\subsection{Positive integer moments}
\label{qSection}
For any positive integer $q\in\mathbb{N}$ we have
\begin{align}
\mathbb{E}[x_{t}^{q}]
&=x_{0}^{q}\mathbb{E}\left[\prod_{i=0}^{n-1}\left(1+\rho e^{\sigma W_{i}-\frac{1}{2}\sigma^{2}t_{i}}\right)^{q}\right]
\\
&=2^{nq}x_{0}^{q}
\mathbb{E}\left[\prod_{i=0}^{n-1}\left(\frac{1}{2}
+\frac{1}{2}e^{\log\rho+\sigma W_{i}-\frac{1}{2}\sigma^{2}t_{i}}\right)^{q}\right]
\nonumber
\\
&=2^{nq}x_{0}^{q}\mathbb{E}\left[\prod_{i=0}^{n-1}e^{(\log\rho+\sigma W_{i}-\frac{1}{2}\sigma^{2}t_{i})Y_{i}}\right],
\nonumber
\end{align}
where $Y_{i}$ are i.i.d. Binomial random variables with parameters $q$ and $\frac{1}{2}$.
Note that a Binomial random variable with parameters $q$ and $\frac{1}{2}$ can be written as a sum
of $q$ i.i.d. Bernoulli random variables. Therefore, we can compute that for any $\theta\in\mathbb{R}$,
\begin{equation}
\lim_{n\rightarrow\infty}\frac{1}{n}\log\mathbb{E}\left[e^{\theta\sum_{i=0}^{n-1}Y_{i}}\right]
=q\log\left(\frac{1}{2}+\frac{1}{2}e^{\theta}\right).
\end{equation}
By Mogulskii theorem,
$\mathbb{P}(\frac{1}{n}\sum_{i=1}^{\lfloor n\cdot\rfloor}Y_{i}\in\cdot)$ satisfies a sample path large deviations principle
with rate function
\begin{equation}
\int_{0}^{1}I_{q}(g'(x))dx,
\end{equation}
where $g(0)=0$, $g$ is absolutely continuous, $0\leq g'\leq q$ and the rate function is $+\infty$ otherwise
and 
\begin{align}
I_{q}(x)&=\sup_{\theta\in\mathbb{R}}\left\{\theta x-q\log\left(\frac{1}{2}+\frac{1}{2}e^{\theta}\right)\right\}
\\
&=q\sup_{\theta\in\mathbb{R}}\left\{\theta\frac{x}{q}-\log\left(\frac{1}{2}+\frac{1}{2}e^{\theta}\right)\right\}
=qI\left(\frac{x}{q}\right)\,.
\nonumber
\end{align}
%where $I_{1}(x)=x\log x+(1-x)\log(1-x)+\log 2$.
Following the proofs for the case $q=1$, we get the following result.

\begin{theorem}\label{qThm}
For any $q\in\mathbb{N}$, $\lambda(\rho,\beta;q):=\lim_{n\rightarrow\infty}\frac{1}{n}\log\mathbb{E}[x_{n}^{q}]$
exists and it can be expressed in terms of a variational formula
\begin{equation}\label{qFormula}
\lambda(\rho,\beta;q)=\sup_{g\in\mathcal{G}_{q}}
\left\{g(1)\log\rho+\beta\int_{0}^{1}(g(1)-g(x))^{2}dx-q\int_{0}^{1}I(g'(x)/q)dx\right\},
\end{equation}
where $I(x)=x\log x+(1-x)\log(1-x)$ and
\begin{equation}
\mathcal{G}_{q}:=\left\{g:[0,1]\rightarrow[0,q], \text{$g(0)=0$, $g$ is absolutely continuous and $0\leq g'\leq q$}\right\}.
\end{equation}
\end{theorem}

\begin{remark}
By replacing $g$ by $q\cdot g$ in \eqref{qFormula}, we can express $\lambda(\rho,\beta;q)$ as
\begin{align}\label{qIdentity}
\lambda(\rho,\beta;q)&=\sup_{g\in\mathcal{G}}
\left\{qg(1)\log\rho+q^{2}\beta\int_{0}^{1}(g(1)-g(x))^{2}dx-q\int_{0}^{1}I(g'(x))dx\right\}
\\
& =q \lambda(\rho,q\beta;1)\,.
%\sup_{g\in\mathcal{G}}
%\left\{g(1)\log\rho+q\beta\int_{0}^{1}(g(1)-g(x))^{2}dx-\int_{0}^{1}I(g'(x))dx\right\},
\nonumber
\end{align}
since $\lambda\geq 0$ and $q\geq 0$, where
\begin{equation}
\mathcal{G}=
\mathcal{G}_{1}=
\left\{g:[0,1]\rightarrow[0,1], \text{$g(0)=0$, $g$ is absolutely continuous and $0\leq g'\leq 1$}\right\}.
\end{equation}
%Hence, we obtain the following identity
%\begin{equation}\label{qIdentity}
%\lambda(\rho,\beta;q)=q\lambda(\rho,q\beta;1).
%\end{equation}
The Euler-Lagrange equation and its solutions, the phase transitions and critical exponents for the case $q=1$
can therefore be directly applied to the general $q\in\mathbb{N}$ case.

The asymptotics \eqref{LargeBeta} and \eqref{LargeRho} yield
\begin{equation}
\lim_{\beta\rightarrow\infty}\frac{\lambda(\rho,\beta;q)}{\beta}=\frac{q^{2}}{3},
\quad
\text{and}
\quad
\lim_{\rho\rightarrow\infty}\left|\lambda(\rho,\beta;q)-\frac{q^{2}\beta}{3}-q\log\rho\right|=0.
\end{equation}
Also, by \eqref{LargeBeta} and \eqref{qIdentity}, we conclude that $\lambda(\rho,\beta;q)$
grows quadratically in $q$ for large $q$,
\begin{equation}
\lim_{q\rightarrow\infty}\frac{\lambda(\rho,\beta;q)}{q^{2}}=\frac{\beta}{3}.
\end{equation}
\end{remark}

\subsection{Lyapunov exponents for linear stochastic recursion}
\label{LinearSection}

The results for the Lyapunov exponent can be generalized to the more 
general linear stochastic recursion (\ref{originalEqn})
$x_{i+1}=a_{i}x_{i}+b_{i}$
where $b_{i}$ are i.i.d. positive random variables independent of $(a_{i})_{i=0}^{\infty}$.
We will show that under very mild conditions on $(b_{i})_{i=0}^{\infty}$, we have the same
Lyapunov exponent as in the $b_{i}\equiv 0$ case.

\begin{theorem}
For fixed $q\in\mathbb{N}$, assume that $\mathbb{E}[b_{0}^{q}]<\infty$. Then,
\begin{equation}
\lim_{n\rightarrow\infty}\frac{1}{n}\log\mathbb{E}[(x_{n})^{q}]=\lambda(\rho,\beta;q).
\end{equation}
\end{theorem}

\begin{proof}
Observe that
\begin{align}
x_{n}&=a_{n-1}x_{n-1}+b_{n-1}
\\
&=a_{n-1}a_{n-2}x_{n-2}+a_{n-1}b_{n-2}+b_{n-1}
\nonumber
\\
&=a_{n-1}a_{n-2}a_{n-3}x_{n-3}+a_{n-1}a_{n-2}b_{n-3}+a_{n-1}b_{n-2}+b_{n-1}
\nonumber
\\
&\cdots\cdots
\nonumber
\\
&=x_{0}\prod_{i=0}^{n-1}a_{i}+b_{0}\prod_{i=1}^{n-1}a_{i}+b_{1}\prod_{i=2}^{n-1}a_{i}
+\cdots+b_{n-2}a_{n-1}+b_{n-1}.
\nonumber
\end{align}
Since $a_{i}\geq 1$ and $b_{i}\geq 0$ for any $i$, by Theorem \ref{qThm},
\begin{equation}
\liminf_{n\rightarrow\infty}\frac{1}{n}\log\mathbb{E}[(x_{n})^{q}]
\geq\liminf_{n\rightarrow\infty}\frac{1}{n}\log\mathbb{E}\left[\left(x_{0}\prod_{i=0}^{n-1}a_{i}\right)^{q}\right]
=\lambda(\rho,\beta;q).
\end{equation}
On the other hand, since $a_{i}\geq 1$, we get
\begin{equation}
x_{n}\leq(x_{0}+b_{0}+b_{1}+\cdots+b_{n-1})\prod_{i=0}^{n-1}a_{i}.
\end{equation}
Since $(b_{i})_{i=0}^{\infty}$ and $(a_{i})_{i=0}^{\infty}$ are independent,
\begin{equation}
\mathbb{E}[(x_{n})^{q}]
\leq\mathbb{E}\left[\left(x_{0}+b_{0}+b_{1}+\cdots+b_{n-1}\right)^{q}\right]
\mathbb{E}\left[\left(\prod_{i=0}^{n-1}a_{i}\right)^{q}\right].
\end{equation}
For any $q\in\mathbb{N}$, since the function $x\mapsto x^{q}$ is convex, by Jensen's inequality, we have
\begin{equation}
\left(\frac{x_{0}+b_{0}+b_{1}+\cdots+b_{n-1}}{n+1}\right)^{q}
\leq\frac{x_{0}^{q}+b_{0}^{q}+b_{1}^{q}\cdots+b_{n-1}^{q}}{n+1}.
\end{equation}
Hence, by Theorem \ref{qThm}, we conclude that
\begin{align}
&\limsup_{n\rightarrow\infty}\frac{1}{n}\log\mathbb{E}[(x_{n})^{q}]
\\
&\leq\limsup_{n\rightarrow\infty}\frac{1}{n}\log\left(\mathbb{E}\left[\left(x_{0}+b_{0}+b_{1}+\cdots+b_{n-1}\right)^{q}\right]
\mathbb{E}\left[\left(\prod_{i=0}^{n-1}a_{i}\right)^{q}\right]\right)
\nonumber
\\
&\leq\limsup_{n\rightarrow\infty}\frac{1}{n}\log\left((n+1)^{q-1}\left[x_{0}^{q}+n\mathbb{E}[b_{0}^{q}]\right]
\mathbb{E}\left[\left(\prod_{i=0}^{n-1}a_{i}\right)^{q}\right]\right)
\nonumber
\\
&=\lambda(\rho,\beta;q).
\nonumber
\end{align}
\end{proof}

\section{Summary and conclusions}

We studied in this paper the distributional properties of a linear stochastic 
recursion of the form $x_{i+1} = a_i x_i + b_i$. The coefficients $a_i$ have 
the form $a_i = 1 + \rho e^{\sigma W_i - \frac12\sigma^2 t_i}$ and have Markovian 
dependence introduced through their dependence on a standard
Brownian motion $W_i$. The $b_i$ are i.i.d. positive definite random numbers.

The main results of the paper concern the rate of growth of the positive integer
moments of the variable $x_n$. We show that the rate of growth of these moments,
or Lyapunov exponents, defined as the limit
\begin{eqnarray}
\lambda_q(\rho,\beta) = \lim_{n\to \infty} \frac{1}{n} \log \mathbb{E}[(x_n)^q]\,,
\quad q \in \mathbb{N}\,,
\end{eqnarray}
exists and is finite as $n\to \infty$, at fixed $\beta = \frac12 \sigma^2 t_n n$. 
The function $\lambda_q(\rho,\beta)$ can be computed explicitly using large deviations theory and is given
by the solution of a variational problem for a functional $\Lambda[f]$. We solve
the variational problem and reduce it to the problem of finding the extremum
of a real function defined in terms of a one-dimensional integral. 
The Lyapunov exponents for $q\in \mathbb{N}$ can be related to  
the Lyapunov exponent of the first moment $q=1$ as shown in Theorem \ref{qThm}. 

The solution of the variational problem shows that the Lyapunov exponents
have non-analytical dependence on the parameters $(\rho,\beta)$, which is similar
to a phase transition in statistical mechanics. For $\rho< \rho_c$, 
below a critical value $\rho_c$, the Lyapunov exponents $\lambda_q(\rho,\beta)$
have discontinuous derivatives along a curve in the $(\rho,\beta)$ plane,
ending at a critical point $(\rho_c, \beta_c)$. Along this curve the Lyapunov
exponents have a first order phase transition, and at the critical point the
transition is second order. 

%A natural question is: what is the probability distribution of the random
%variable $x_n$ as $n\to \infty$? We show in Section~\ref{ASSection} that
%taking the $n\to \infty$ limit at fixed $\beta = \frac12\sigma^2 t_n n$, the 
%random variable
%$\frac{1}{n}\log x_n$ converges almost surely to a constant $\log(1+\rho)$.

The variational problem resulting from the application of large deviations
theory has a direct physical interpretation, as the thermodynamical potential
of a one-dimensional gas of particles, interacting with attractive two-body
interaction energy given by the covariance function of the standard Brownian
motion. The results of our paper give an exact solution of this statistical
mechanics problem in the grand canonical ensemble, and demonstrate the presence 
of a phase transition in this system. The thermodynamical properties
are in agreement with the solution of a lattice gas with the same interaction 
in the thermodynamical limit, obtained in the isobaric-isothermal ensemble 
\cite{DPII}.

As mentioned, the random multiplicative process for $x_n$ can be interpreted
as the grand partition function of a 1-dimensional lattice gas placed in a 
random external field given by a standard Brownian motion. This is a lattice 
gas equivalent of the systems studied in \cite{BK,CM,CMY} in the continuous case.
Our results give the asymptotics of the moments of the grand partition function,
which could be used to determine the properties of the disordered system in the
quenched approximation by an application of the replica approach \cite{MPV}.

We note that similar moment explosions have been studied for the solutions of the
diffusion and Schr\"odinger equations in a random medium given by the
square of a Gaussian random field \cite{ADLM,MCL}. 

The results of this paper have applications to the numerical simulation of 
stochastic differential equations. Euler discretization of certain stochastic 
differential equations gives linear stochastic recursions of the form considered 
here. The weak convergence of the discretization requires the uniform boundedness
of the moments of the discretized variable, see for example \cite{Alfonsi}. 
Similar recursions are obtained when considering the Euler discretization of
stochastic volatility models with log-normally distributed volatility. 
The results of this paper give explicit results for the growth rate of the
moments of the stochastic variable $x_n$ in such models. We will present
detailed application of these results and methods to models of practical
interest in future work.

%%%%%%%%%%%%%%%%%%%%%%%%%%%%%%%%%%%%%%%%%%%%%%%%%%%%%%%%%%%%%
\appendix

\section{Solution of the Euler-Lagrange equation}
\label{app1}

%The equation (\ref{LE4}) can be alternatively expressed in terms of the function
%\begin{eqnarray}
%h(y) = \log\frac{f(y)}{1-f(y)}\,.
%\end{eqnarray}

We present in this Appendix the solution of the Euler-Lagrange
equation (\ref{heq}) for $h(y)$. 
It is useful to introduce the notation
\begin{eqnarray}\label{Vdef}
V(h) = 2\beta \log(1 + e^h).
\end{eqnarray}
In terms of this function, the equation (\ref{heq}) is written as
\begin{eqnarray}
h''(y) = - V'(h(y))\,,
\end{eqnarray}
which is analogous to Newton's law for a particle moving in the potential
$V(h)$. Written in this form, it is easy to check that the following combination
is a constant of motion of the equation (\ref{heq}) 
\begin{eqnarray}\label{Edef}
E = \frac12 (h'(y))^2 + V(h(y)) = V(h(1))\,.
\end{eqnarray}
The value of the constant of motion was determined from the boundary condition 
(\ref{hbc}) at $y=1$ as
$E  = V(h(1))$. 

The relation (\ref{Edef}) will be useful to express $h'(y)$ in
terms of $h(y)$.  In particular, we can use it to write
\begin{eqnarray}
\frac{dy}{dh} = \frac{1}{\sqrt{2(E-V(h(y)))}}  = \frac{1}{2\sqrt{\beta}}
\frac{1}{\sqrt{\log\frac{1+e^{h(1)}}{1+e^h}}}.
\end{eqnarray}
Integrating this relation over $h$ from $h(0)$ to $h(1)$ we get an equation
for $h(1)$
\begin{eqnarray}
1 = \frac{1}{2\sqrt{\beta}}
\int_{h(0)=\log\rho}^{h(1)} \frac{dx}{\sqrt{\log\frac{1+e^{h(1)}}{1+e^x}}}.
\end{eqnarray}

Thus we can find $h(1)$ by solving the equation
\begin{eqnarray}\label{h1eq2}
F(h(1);\rho) \equiv \int_{\log\rho}^{h(1)}
\frac{dx}{\sqrt{ \log\frac{1+e^{h(1)}}{1+e^x}}} = 2\sqrt{\beta},
\end{eqnarray}
%where $h(1)$ satisfies the equation
%\begin{eqnarray}\label{h1eq2}
%F(h(1);\rho) = 2\sqrt{\beta}\,
%\end{eqnarray}
%where
%\begin{eqnarray}\label{Fdef2}
%F(a;\rho) \equiv \int_{\log\rho}^{a}
%\frac{dx}{\sqrt{ \log\frac{1+e^a}{1+e^x}}}.
%\end{eqnarray}

Once $h(1)$ is found, the complete shape of the function $h(y)$ can be
determined by solving the equation
\begin{eqnarray}
\int_{\log\rho}^{h(y)}
\frac{dx}{\sqrt{ \log\frac{1+e^{h(1)}}{1+e^x}}} = 2\sqrt{\beta} y \,.
\end{eqnarray}

As the next step in the solution of the variational problem we would like 
to compute the functional $\Lambda[f]$ corresponding to a solution $f(x)$ 
of the Euler-Lagrange equation (\ref{LE4}) with boundary condition (\ref{BC}). 
If the Euler-Lagrange equation has a unique solution $f(y)$, then 
$\lambda(\rho,\beta) = \Lambda[f]$. However, if it has several solutions,
as is the case around the phase transition, the Lyapunov exponent
is given by the supremum
of $\Lambda[f]$ over these multiple solutions
\begin{eqnarray}
\lambda(\rho,\beta) = \mbox{sup}_{f} \Lambda[f]\,.
\end{eqnarray}
The main result is summarized in Proposition~\ref{propLambda}. We present 
here the proof of this result.

\begin{proof}
We start by expressing the functional $\Lambda[f]$ in a simpler
form as
\begin{equation}\label{Lambdadef}
\Lambda[f] = \frac12 \log\rho \int_0^1 dx f(x) - 
\frac12 \int_0^1 dx f(x) \log\frac{f(x)}{1-f(x)} - 
\int_0^1 dx \log(1-f(x))\,.
\end{equation}
This is obtained by eliminating the double integral over $K(z,y)$ using the
Euler-Lagrange equation (\ref{EulerLagrange}). Multiplying this equation with 
$f(y)$ and integrating over $y$, this allows us to solve for the double
integral. Substituting into (\ref{symm}) gives the result above.

There are three integrals appearing in $\Lambda[f]$. We will show next that 
their sum can be expressed in terms of $h(1)$ alone. We consider them in turn.

The first integral is related to $f'(0)$ as 
\begin{eqnarray}\label{18}
f'(0) = 2\beta f(0)(1-f(0)) \int_0^1 dz f(z) = \frac{2\beta\rho}{(1+\rho)^2}
\int_0^1 dz f(z)\,.
\end{eqnarray}
This is obtained by taking $y=0$ in (\ref{ELp}).
%This is always positive. The inequality $f(z) < 1$ implies that the slope $f'(0)$
%is bounded from above as
%\begin{eqnarray}\label{fpbound}
%f'(0) \leq \frac{2\beta\rho}{(1+\rho)^2}\,.
%\end{eqnarray}
This is uniquely determined by $h(1)$, as can be seen from equation (\ref{Edef})
\begin{eqnarray}
\frac12 (h'(y))^2 + V(h(y)) = V(h(1))
\end{eqnarray}
and thus $h'(y) = \sqrt{2(V(h(1)) - V(h(y))} = 
2\sqrt{\beta} \sqrt{\log\frac{1+e^{h(1)}}{1+e^{h(y)}}}$. Taking here $y=0$
we can express $h'(0)$ in terms of $h(1)$.
As a result we have
\begin{equation}
I_0 = \int_0^1 dx f(x) = \frac{1}{2\beta} h'(0) = \frac{1}{\sqrt{\beta}}
\sqrt{\log\frac{1+e^{h(1)}}{1+\rho}}\,.
\end{equation}
%The bound $f(x) \leq 1$ implies $\int_{0}^{1} dx f(x)\leq 1$, which gives an 
%upper bound on $h(1)$, which is equivalent
%to the bound (\ref{fpbound}) on $f'(0)$.
%We get that the range of variation of $h(1)$ is 
%\begin{eqnarray}\label{h1bound}
%\log\rho \leq h(1) \leq \log[(1+\rho) e^\beta - 1]\,.
%\end{eqnarray}

The second integral in $\Lambda[f]$ is
\begin{align}
I_1 &=\int_0^1 dx f(x) \log \frac{f(x)}{1-f(x)}
= \int_0^1 dx h(x) \frac{e^{h(x)}}{1+e^{h(x)}} 
\\
&= -\frac{1}{2\beta} \int_0^1 dx h(x) h''(x) 
= \frac{1}{2\beta}  \Big( \int_0^1 dx (h'(x))^2 + h(0) h'(0) \Big), \nonumber
\end{align}
where we integrated by parts in the last step and used $h'(1) = 0$.

Finally, the third integral is
\begin{align}
I_2 &=\int_0^1 dx \log(1-f(x)) 
= - \int_0^1 dx \log(1 + e^{h(x)})
\\
&= -\frac{1}{2\beta} \int_0^1 dx V(h(x)) 
= -\frac{1}{2\beta} \int_0^1 dy \Big( E - \frac12 (h'(y))^2 \Big) \nonumber
\\
&= - \log(1 + e^{h(1)}) + \frac{1}{4\beta} \int_0^1 dy (h'(y))^{2}, \nonumber
\end{align}
where $V(h)$ is defined in (\ref{Vdef}). In the second line we used
(\ref{Edef}) to eliminate $V(h(y))$ in terms of $(h'(y))^2 = 4\beta 
\log\frac{1+e^{h(1)}}{1+e^{h(y)}}$.

The integrals $I_1$ and $I_2$ appear in $\Lambda[f]$ in the combination
\begin{align}
\frac12 I_1 + I_2 
&= \frac{1}{2\beta} \int_0^1 dx \Big\{
\frac12 (h'(x))^2 - V(h(x)) \Big\} + \frac{1}{4\beta} h(0) h'(0) \\
&= 
\frac{1}{\sqrt{\beta}}
\int_{\log\rho}^{h(1)} dx \sqrt{\log\frac{1+e^{h(1)}}{1+e^x}} - 
\log(1 + e^{h(1)}) + \frac{1}{4\beta} \log\rho h'(0)\,.\nonumber
\end{align}

Substituting this into (\ref{Lambdadef}) the last term cancels against
the term proportional to the integral $I_0$, and we obtain the result 
(\ref{Lambdah1}). 
This concludes the proof of (\ref{Lambdah1}).
\end{proof}

\section{Analytical solution for $\beta d^2 \gg 1$}\label{app2}

We present in this Appendix an analytical solution of the 
variational problem for $\Lambda(d)$ in the $\beta d^2 \to \infty$ limit. 
This is used to derive the properties of the phase transition in the same 
limit, which are summarized in Proposition~\ref{largebeta}.

We will prove here the following approximation for the functional $\Lambda(d)$
\begin{eqnarray}\label{Lambdabound}
&& \Lambda(d) = 
 \beta d^2 - \frac23\beta d^3 + d\log\Big(\frac{\rho}{1+\rho}\Big)
 + \log(1+\rho)+ o(a^{-1})\,,
\end{eqnarray}
where $a = \beta d^2$.

The starting point is the observation that the integral appearing 
in (\ref{Lam2}) depends only on the combination $a=\beta d^2$. We denote it as
\begin{eqnarray}\label{Jdef}
J(a;\rho) = \int_0^1 \frac{y^2 dy}{1+\rho - e^{a(y^2-1)}}
\end{eqnarray}
We would like to obtain an approximation for this integral for $a \gg 1$.
This is given by the following result.

\begin{lemma}
The integral $J(a;\rho)$ has the following expansion for $a\gg 1$ 
\begin{eqnarray}\label{Jasympt}
 J(a;\rho) = \frac{1}{1+\rho}
\Big\{
\frac13 - \frac{1}{2a} \log\Big(\frac{\rho}{1+\rho}\Big) + o(a^{-2})
\Big\}\,.
\end{eqnarray}
\end{lemma}

\begin{proof}

This result is shown by proving matching lower and upper bounds 
for the integral $J(a;\rho)$.
We start by deriving a lower bound for the integral. 
This bound follows from using the inequality $y^2 - 1 \geq 2(y-1)$,
which holds for any $y\in (0,1)$,
in the exponent in the denominator of the integral (\ref{Jdef}).
Then the integral can be performed exactly with the result 
\begin{eqnarray}\label{Jlower}
&& J(a;\rho) \geq \frac{1}{3a^3(1+\rho)}
\Big\{
a^3 - \frac32 a^2 \log\Big(\frac{\rho}{1+\rho}\Big) \\ 
&& \qquad - \frac32 a \mbox{Li}_2
\Big(\frac{1}{1+\rho}\Big) + \frac34 \mbox{Li}_3
\Big(\frac{1}{1+\rho}\Big) - \frac34 \mbox{Li}_3
\Big(\frac{e^{-2a}}{1+\rho}\Big) \Big\} \nonumber
\end{eqnarray}
Here $\mbox{Li}_n(z)$ denotes the polylogarithm of 
order $n$ \cite{AS}.

For $\rho\to 0$ and $a\gg 1$ the polylogarithms approach constant values
\begin{eqnarray}
&& \lim_{\rho \to 0} 
\mbox{Li}_2 \Big(\frac{1}{1+\rho}\Big) = \frac{\pi^2}{6} \\
&& \lim_{\rho \to 0} 
\mbox{Li}_3 \Big(\frac{1}{1+\rho}\Big) = \zeta(3) \simeq 1.202 \\
&& \lim_{x \to 0} 
\mbox{Li}_3 (x) = x + o(x^2) \,.
\end{eqnarray}
We get thus the lower bound
\begin{eqnarray}\label{Jlower2}
J(a;\rho) \geq \frac{1}{3(1+\rho)} -\frac{1}{2a(1+\rho)} 
\log\Big( \frac{\rho}{1+\rho}\Big) + o(a^{-2})
\,.
\end{eqnarray}

Next we prove the upper bound for the integral $J(a;\rho)$
\begin{eqnarray}\label{Jupper}
J(a;\rho) \leq \frac{1}{1+\rho} \Big\{ \frac13 - \frac{1}{2a}
\log\left(\frac{\rho}{1+\rho - e^{-a}}\right) \Big\} \,.
\end{eqnarray}
This is obtained by considering the difference
\begin{eqnarray}
\frac{1}{1 + \rho - e^{a(y^2-1)}} - \frac{1}{1+\rho} = 
\frac{1}{(1+\rho)[(1+\rho)e^{a(1-y^2)}-1]}
\end{eqnarray}
Multiplying with $y^2$ and integrating over $y:(0,1)$ we have
\begin{eqnarray}
&& J(a;\rho) - \frac{1}{3(1+\rho)} \leq \frac{1}{1+\rho}
\int_0^1 \frac{ydy}{(1+\rho) e^{a(1-y^2)}-1} \\
&& = -\frac{1}{2a(1+\rho)} \log\Big( \frac{\rho}{1+\rho-e^{-a}}\Big)
< -\frac{1}{2a(1+\rho)} \log\Big( \frac{\rho}{1+\rho}\Big)
\,.\nonumber
\end{eqnarray}
This proves the upper bound (\ref{Jupper}). Comparing with the
lower bound (\ref{Jlower2}) we obtain the expansion (\ref{Jasympt}).
\end{proof}

The final result (\ref{Lambdabound}) follows directly from using the
expansion (\ref{Jasympt}) of the integral $J(a;\rho)$ into the expression
for $\Lambda(d)$.

We use next the approximation (\ref{Lambdabound}) to study the solution of the
variational problem for $\Lambda(d)$ for $\beta \to \infty$. 
We would like to find the supremum over $d$ of the cubic polynomial in
(\ref{Lambdabound}). This supremum is reached at $d = 0$ or 
$d_* = \frac12\Big(1 +\sqrt{1+\frac{2}{\beta}\log\frac{\rho}{1+\rho}}\Big) > 0$. 
according to the following condition 
\begin{eqnarray}
\mbox{sup}_d \Lambda(d) = \left\{
\begin{array}{cc}
\Lambda(0) = \log(1+\rho) & \mbox{ if } \beta \leq - \frac83\log \frac{\rho}{1+\rho} \\
\Lambda(d_*) >\log (1+\rho) & \mbox{ if } \beta > - \frac83\log\frac{\rho}{1+\rho} \\
\end{array}
\right.
\end{eqnarray}

The supremum switches branches at the point
\begin{eqnarray}
\beta_{\rm cr}(\rho) = - \frac83\log \Big(\frac{\rho}{1+\rho}\Big)
\end{eqnarray}
which thus defines the phase transition curve for $\beta \to \infty$. 
This proves the result (\ref{ptlimit}).
As $\beta\to \infty$ or equivalently $\rho \to 0$, the slope of this curve 
in $(-\log\rho,\beta)$  coordinates 
approaches the value $8/3$. This is smaller than the slope of the phase 
transition curve in the mean-field approximation, which is equal to 3.

The values of $d_1, d_2$ along the phase transition curve approach
\begin{eqnarray}
d_1  = 0\,,\quad 
d_2  = \frac34 = 0.75
\end{eqnarray}
as $\beta \to \infty$.
This proves the result (\ref{d2largebeta}). For $\beta \gg \beta_{\rm cr}(\rho)$,
we have $\lim_{\beta\to \infty} d = 1_-$ and the Lyapunov exponent becomes
\begin{eqnarray}
\lim_{\beta\to \infty} \lambda(\rho,\beta) =
\lim_{\beta\to \infty} \Lambda(d) = \frac13 \beta + \log\rho\,.
\end{eqnarray}
This result is in agreement with the large $\beta$ asymptotic behavior of the 
Lyapunov exponent proven in Proposition~\ref{prop:bounds}, 
see Eq.~(\ref{LargeBeta}).

%%%%%%%%%%%%%%%%%%%%%%%%%%%%%%%%%%%%%%%%%%%%%%%%%%%%%%%

\section*{Acknowledgements}
We would like to thank the Editor and two anonymous referees for useful
comments and advice. 


\begin{thebibliography}{99}


\bibitem{AS}
M.~Abramowitz and I.~A.~Stegun.
Handbook of Mathematical Functions with Formulas, Graphs and 
Mathematical Tables.
Dover Publications, New York (1972).

\bibitem{Alfonsi}
A.~Alfonsi, High order discretization schemes for the CIR process:
application to Affine Term Structure and Heston models.
\textit{Math.~Comp.} 
\textbf{79}, 209-237 (2010).

\bibitem{AristoffZhu}
D.~Aristoff and L.~Zhu.
On the phase transition curve in a directed exponential random graph model. 
arXiv:1404.6514[math.PR]. (2014).

\bibitem{ADLM}
A.~Asselah, P.~Dai Pra, J.~L.~Lebowitz and Ph.~Mounaix,
Diffusion effects on the breakdown of a linear amplifier model
driven by the square of a Gaussian field.
\textit{J.~Stat.~Phys.}
\textbf{104}, 24-32 (1990).


\bibitem{BDT}
F.~Black, E.~Derman and W.~Toy.
A one-factor model of interest rates and its application to treasury bond options.
\textit{Financial Analysts Journal}
\textbf{46}, 24-32 (1990).

\bibitem{BG}
J.~P.~Bouchaud and A.~Georges, 
Anomalous diffusion in disordered media: Statistical mechanisms, models
and physical applications.
\textit{Physics Reports}
\textbf{195}, 127-293 (1990).

\bibitem{BK}
K.~Broderix and R.~Kree. 
Thermal equilibrium with the Wiener potential: Testing the replica variational
approximation.
\textit{Europhys.~Lett.}   
{\bf 32}, 343 (1995)


\bibitem{Chatterjee}
S.~Chatterjee and P.~Diaconis. 
Estimating and understanding exponential random graph models.
\textit{Annals of Statistics}.
\textbf{41}, 2428-2461 (2013).

\bibitem{CN}
J.~E.~Cohen and C.~M.~Newman. 
The stability of large random matrices and their products. 
\textit{The Annals of Probability}. 
{\bf 12}, 283-310 (1984).

\bibitem{CM}
A.~Comtet and C.~Monthus. 
On the flux distribution in a one-dimensional disordered system. 
\textit{J.~Phys.} I  
{\bf 4}, 635 (1994)

\bibitem{CMY}
A.~Comtet, C.~Monthus and M.~Yor. 
Exponential functionals of Brownian motion and disordered systems. 
\textit{J.~Appl.~Prob.} 
{\bf 35}, 255-271 (1998)

\bibitem{Dembo} 
A.~Dembo and O.~Zeitouni. 
\textit{Large Deviations Techniques and Applications}, 2nd Edition, Springer, New York, 1998.

\bibitem{Dufresne1}
D.~Dufresne. 
The integral of geometric Brownian motion.
\textit{Adv.~Appl.~Prob.}. 
{\bf 33}, 223-241 (2001).

\bibitem{Dufresne2}
D.~Dufresne. 
The log-normal approximation in financial and other computations.
\textit{Adv.~Appl.~Prob.}. 
{\bf 36}, 747-773 (2004).


\bibitem{Ellis}
R.~Ellis.
\textit{Entropy, Large Deviations, and Statistical Mechanics} (Classics in Mathematics).
Springer, New York, 2005.


\bibitem{Goldie}
C.~M.~Goldie. 
Implicit renewal theory and tails of solutions of random equations.
\textit{Ann.~Appl.~Prob}. {\bf 1}, 126 (1991).

%\bibitem{SABR}
%P.~S.~Hagan {\em et al}
%Managing smile risk.
%\textit{Wilmott}
%84 (2000).

%\bibitem{HW}
%J.~Hull and A.~White.
%The pricing of options on assets with stochastic volatility.
%\textit{The Journal of Finance}
%\textbf{XLII}(2), 281-300 (1987).


\bibitem{Kadanoff}
L.~P.~Kadanoff. 
\textit{Statistical Physics: Statics, Dynamics and Renormalization}.
World Scientific, Singapore, 2000.

\bibitem{Kesten}
H.~Kesten. Random difference equations and renewal theory for products of random matrices. 
\textit{Acta Math}. {\bf 131}, 207 (1973).

\bibitem{MS} J.~Messer and H.~Spohn. 
Statistical Mechanics of the Isothermal Lane-Emden Equation. 
\textit{J.~Stat.~Phys.}
\textbf{29}, 561 (1982). 

\bibitem{MPV}
M.~M\`ezard, G.~Parisi and M.~Virasoro,
\textit{Spin glass theory and beyond}.
Singapore, World Scientific (1987).

\bibitem{LP}
J.~Lebowitz and O.~Penrose.
Rigorous treatment of the van der Waals-Maxwell theory of the liquid-vapor transition.
\textit{J. Math. Phys.}
\textbf{7}, 98 (1966).

\bibitem{Lewontin}
R.~C.~Lewontin and D.~Cohen.
On population growth in a randomly varying environment.
\textit{Proc. Natl. Acad. Sci. USA}
\textbf{62}, 1056 (1969).

\bibitem{MST}
T.~Mikosh, G.~Samorodnitsky and L.~Tafakori.
Fractional moments of solutions to stochastic recurrence equations.
\textit{J.~Appl.~Prob.}
\textbf{50}, 969-982 (2013).

\bibitem{Mitzenmacher}
M.~Mitzenmacher.
A brief history of generative models for power law and log-normal distributions.
\textit{Internet Math.}
\textbf{1}, 226-251 (2004).

\bibitem{MCL}
P.~Mounaix, P.~Collet and J.~L.~Lebowitz, 
Propagation Effects on the Breakdown of a Linear Amplifier Model: 
Complex-Mass Schr\"odinger Equation Driven by the Square of a Gaussian Field.
\textit{Comm.~Math.~Phys.}
\textbf{264}, 741-758 (2006).



\bibitem{DPI}
D.~Pirjol.
Emergence of heavy tailed distributions in a random multiplicative model
driven by a Gaussian stochastic process.
\textit{J. Stat. Phys.}
\textbf{154}, 781-806 (2014).

\bibitem{DPII}
D.~Pirjol.
Long term growth rate in a random multiplicative model.
\textit{J. Math. Phys.}
\textbf{55}, 083305 (2014).

\bibitem{Radin}
C.~Radin and M.~Yin.
Phase transitions in exponential random graphs.
\textit{Annals of Applied Probability}.
\textbf{23}, 2458-2471 (2013).

\bibitem{Ruelle}
D.~Ruelle. 
Analyticity properties of the characteristic exponents of random matrix products. 
\textit{Adv.~Math.} {\bf 32} 68-80 (1979).

\bibitem{Roitershtein}
A.~Roitershtein. 
One-dimensional linear recursions with Markov-dependent coefficients. 
\textit{Ann.~Appl.~Prob}. {\bf 17}, 572 (2007).

\bibitem{Saporta}
B.~DeSaporta. 
Tail of the stationary solution of the stochastic equation $Y_{n+1}=a_nY_n+b_n$ with
Markovian coefficients.
\textit{Stoch.~Proc.~Appl.}. {\bf 115}, 1954 (2005).


\bibitem{Sornette}
D.~Sornette and R.~Cont.
Convergent multiplicative processes repelled from zero: power laws and truncated power laws.
\textit{J. Phys.}
\textbf{I 7}, 431 (1997).

\bibitem{Stanley}
E.~H.~Stanley.
Introduction to Phase Transitions and Critical Phenomena.
Oxford University Press, 1987.

\bibitem{VaradhanII} 
S.~R.~S.~Varadhan. 
\textit{Large Deviations and Applications}, SIAM, Philadelphia, 1984.

\bibitem{Vervaat}
W.~Vervaat. 
On a stochastic difference equation and a representation of nonnegative infinitely 
divisible random variables.
\textit{Adv.~Appl.~Prob.}
\textbf{11}, 750-783 (1979).

\bibitem{Yor}
M.~Yor. 
On some exponential functions of Brownian motion.
\textit{Adv.~Appl.~Prob.}
\textbf{24}, 509-531 (1992).


\end{thebibliography}


\end{document}